\RequirePackage{fix-cm}
\documentclass[twocolumn]{svjour3}          
\smartqed  
\usepackage{graphicx}
\usepackage{mathptmx}      
\usepackage{amsmath,amsfonts}
\usepackage{amsrefs}
\usepackage{verbatim}

\usepackage[normalem]{ulem}

\DeclareMathOperator{\diag}{diag}
\newcommand\cocoa{{\hbox{\rm C\kern-.13em o\kern-.07em C\kern-.13em o\kern-.15em A}}}
\newcommand{\design}{\mathcal D}
\newcommand{\expect}{\operatorname{\mathbb E}}
\newcommand{\expectof}[1]{\expect\left(#1\right)}
\newcommand{\normof}[1]{\left\Vert#1\right\Vert}
\newcommand{\reals}{\mathbb R}
\newcommand{\scalarat}[3]{\left\langle#2,#3\right\rangle_{#1}}

\begin{document}

\title{The algebra of interpolatory cubature formul{\ae} for generic nodes}
\titlerunning{Algebra of Cubature}

\author{Claudia Fassino \and Giovanni Pistone \and Eva Riccomagno}

\authorrunning{C Fassino \and G Pistone \and E Riccomagno}

\date{Submitted \today}

\institute{Claudia Fassino\and Eva Riccomagno \at Dipartimento di Matematica, Universit\`a di Genova, \\
Via Dodecaneso 35, 16141 Genova, Italy \\ \email{fassino@dima.unige.it} \\   \email{riccomagno@dima.unige.it} \and Giovanni Pistone \at Collegio Carlo Alberto, Via Real Collegio 30, 10024 Moncalieri, Italy \\ \email{giovanni.pistone@carloalberto.org} 
}

\maketitle

\begin{abstract}
We consider the classical problem of computing the expected value of a real function $f$ of the $d$-variate random variable $X$ using cubature formul\ae. 
We use in synergy tools from Commutative Algebra for cubature rul\ae, from elementary orthogonal polynomial theory and from Probability. 
\end{abstract}



\section{Introduction} 

 Consider the classical problem of computing the expected value of a real function $f$ of the $d$-variate random variable $X$ as a linear combination of its values $f(z)$ at a finite set of points $z \in \design \subset \reals^d$. The general cubature problem is: determine classes of functions $f \colon \reals^d \to \reals$, finite set of $n$ \emph{nodes} $\design \subset \reals^d$ and positive weights $\{\lambda_z\}_{z \in \design}$ such that
 \begin{align}\label{eq:intro}
 \expect(f(X))= \int_{\reals^d}f(x)\, d\lambda(x) = \sum_{z \in \design} f(z) \lambda_z
 \end{align}
  where $\lambda$ is the probability distribution of the random vector $X$. In the univariate case, $d=1$, the set $\design$ is the set of zeros of a \emph{node polynomial}, e.g. the $n$-th orthogonal polynomial for $\lambda$, see e.g. \cite[Sec.~1.4]{MR2061539}. Not much is known in the multivariate case, unless the set of nodes is product of one-dimensional sets. 

 A similar setting appears in statistical Design of Experiment (DoE) where one considers a finite set of \emph{treatments} $\design$ and the experimental outputs as function of the treatment. The set of treatments and the set of nodes are both described efficiently as zeros of systems of polynomial equations, i.e. as what is called a 0-\emph{dimensional variety} in Commutative Algebra. This framework is systematic for Algebraic Statistics where  tools from modern Computational Commutative Algebra are used to address problems in statistical inference and modelling, see e.g. 
 \cites{DrtonSullivantSturmfelsLectureNotes,MR2332740,MR2640515}. In DoE the set $\design$ is called a \emph{design} and the \emph{affine} structure of the ring of real functions on $\design$ is analyzed in detail because it represents the set of real responses to treatments in $\design$. However,  in the algebraic setting the \emph{euclidean} structure, such as the computation of mean values, is missing. In algebraic design of experiment the computation of mean values has been obtained by considering very special sets called \emph{factorial designs}, e.g. $\{+1,-1\}^d$, see e.g. \cite{MR1772046} and \cite[Ch.~5]{MR2332740}. Note that $\{+1,-1\}$ is the zero set of the polynomial $x^2-1$. 

 The purpose of the present paper is to discuss how the above comes together by considering orthogonal polynomials. In particular, we consider algorithms from Commutative Algebra for the cubature problem in \eqref{eq:intro} by mixing tools from elementary orthogonal polynomial theory and Probability. Vice versa, Formula \eqref{eq:intro} provides an interesting interpretation of the RHS term as expected value.

 We proceed by steps of increasing degree of generality. 
 In Section~\ref{section:OrthogonalPolynomialsOnTheLine} we consider the univariate case and take $\lambda$ to admit an orthogonal system of polynomials. Let $g(x) = \prod_{z\in \design}(x-z)$
 and by univariate division given a polynomial $p$ there exist unique $q$ and $r$ such that $p=q \ g+r$ and $r$ has degree smaller than the number of points in $\design$, that is the degree of $g$. Furthermore, $r$ can be written as $\sum_{z\in \design} r(z) l_z(x)$ where $l_z$ is the Lagrange polynomial for $z\in \design$. Then we show that
 \begin{enumerate}
 \item the expected values of $p$ and $r$ coincide if and only if the $n$-coefficients of the Fourier expansion of $q$ with respect to the orthogonal polynomial system is zero, 
 \item the weights $\lambda_z$ in \eqref{eq:intro} are the expected values of the Lagrange polynomials $l_z$, for $z\in \mathcal D$.
 \end{enumerate}
 The case when the design $\mathcal D$ is a proper subset of the zero set of the $n$-th orthogonal polynomial is developed in Section~\ref{fraction_section}.
 
In Section~\ref{Hermite:section}, $\lambda$ is a standard Gaussian probability law and $\design$ the zero set of the $n$-th Hermite polynomial $H_n$.
By applying Stein-Markov theory we give  
 a representation of some Hermite polynomials, including those of degree $2n - 1$, as
 sum of an element in the polynomial ideal generated by $H_n$  and of a reminder. See Theorem \ref{productHermite} and the following discussion, in particular Equation~\eqref{peccati} which, unsurprisingly, is reminiscent of other formul\ae \, for iterated It\={o} integrals e.g.~\cite[Eq.~(6.4.17)]{peccatitaqqu}.
 The point is to describe a ring structure of the space generated by Hermite polynomials up to a certain order.
 This ring structure is essentially the aliasing on functions induced by limiting observations to $\design$. 
 The particular form of the recurrence relationship for Hermite polynomials makes this possible and we suspect that the study of the ring structure over $\design$ for other systems of orthogonal polynomials will require different tools from those we use here. 
 
 This result implies a system of equations in Theorem~\ref{weighing_theorem} (extended to the multidimensional case in Section~\ref{section:HigherDimension})
 which gives an implicit description of design and weights via two polynomial equations. 
 We envisage applicability of this in the choice of $\design$ for suitable classes of functions but have not developed this here.    
 
 Section~\ref{section:genericNoRing} contains our most general set-up: we restrict  ourselves to product probability measures on $\reals^d$ but consider any set of $n$ distinct points in $\reals^d$.
 Then a Buchberger-M\"oller type of algorithm is provided that works exclusively over vector-space generated by orthogonal polynomials up to a suitable degree. It gives a generating set of the vanishing ideal of $\design$ expressed in terms of orthogonal polynomials. 
 This is used to determine sufficient and necessary conditions on a polynomial function $f$ for which (1) holds for suitably defined weights. 
 Furthermore, exploiting the Fourier expansion of a Gr\"obner basis of the vanishing ideal of $\design$, some results about the exactness of the cubature formul\ae \, are shown.
 Of course it will be of interest to determine generalisations of our results to the cases where $\lambda$ is not a product measure and still admits an orthogonal system of polynomials.  
   

\subsection{Basic commutative algebra}\label{section:generalframework}

We start with some notation on polynomials: $\reals[x]$ is the ring of polynomials with real coefficients and in the $d$-variables (or indeterminate) $x=(x_1,\ldots,x_d)$; for $\alpha=(\alpha_1,\ldots,\alpha_d)$ a $d$-dimensional vector with non-negative integer entries, $x^\alpha=x_1^{\alpha_1} \ldots x_d^{\alpha_d}$ indicates a monomial; $\tau$ indicates a term-ordering on the monomials in $\reals[x]$. 
 If $d=1$ there is only one term ordering, this is  not the case for $d \ge 2$. Designs of product form share some commonalities with the one dimension case. Because of this, term orders are not much used in standard quadrature theory. We will see that refining the division partial order to a proper term-order is actually relevant in some, but not all, multivariate cases.

The total degree of the monomials $x^\alpha$ is $\sum_{i=1}^d \alpha_i$. The symbol $\reals[x]_k$ indicates the set of polynomials of total degree $k$ and $\reals[x]_{\leq k}$ the vector space of all polynomials of at most total degree $k$. Let $\design$ be a finite set of distinct points in $\reals^d$, $\lambda$ a probability measure over $\reals^d$ and $X$ a real-valued random vector with probability distribution $\lambda$ so that the expected value of the random variable $f(X)$ is $\expect(f(X)) = \int f(x) d\lambda(x)$. 

Given a term ordering $\tau$,  let $f_1,\ldots,f_t\in \reals[x]$ form a \emph{Gr\"ob\-ner basis} with respect to  $\tau$ (see~\cite[Ch. 2]{MR2290010}) of the ideal $\mathcal I(\mathcal D)$ of polynomials vanishing over $\design$. For each $p\in \reals[x]$ there exist $h_i$, $i=1,\dots,t$ and an unique $r\in \reals [x]$ such that
\begin{align}  \label{eq:1}
p(x) &= \sum_{i=1}^t h_i(x)f_i(x) + r(x), 
\end{align}
and $r$ has its largest term in $\tau$ not divisible by the largest term of $f_i$, $i=1,\ldots,t$. 
Note that the $h_i(x) \in \reals[x]$ are \emph{not} necessarily unique. For all $p\in \reals[x]$, the polynomial $r$ above is referred to as \emph{reminder} or \emph{normal form}. It is often indicated with the symbol $\operatorname{NF}_\tau(p,\{ f_1,\ldots,f_t \})$, or the shorter version $\operatorname{NF}(p)$, while $\langle f_1,\ldots,f_t \rangle$ indicates the polynomial ideal generated by $f_1,\ldots,f_t$. 
Moreover, monomials not divisible by the largest terms of $f_i$, $i=1,\ldots,t$, form a vector basis of monomial functions for the vector space  $\mathcal L(\design)$ of real functions on $\design$.  
The polynomials $g=\sum_{i=1}^t h_i f_i$ and $r$ in \eqref{eq:1} are fundamental in the applications of Algebraic Geometry to finite spaces. 
Various general purpose softwares, including Maple, Mathematica, Matlab and computer algebra softwares, like \cocoa, Macaulay, Singular, allow manipulation with polynomial ideals, in particular compute reminders and monomial bases.

The polynomial  $r(x)$ in \eqref{eq:1} can be written uniquely as
\begin{equation}\label{indicator-reminder}
r(x) = \sum_{z\in \design} p(z) l_z(x)  
\end{equation} 
where $l_z$ is the indicator polynomial of the point $z$ in $\design$, i.e for $x\in \design$ it is $l_z(x)=1$ if $x=z$ and $l_z(x)=0$ if $x\neq z$. Equation~(\ref{indicator-reminder}) follows from the fact that $\{ l_z: z\in \design\}$ is a $\reals$-vector space basis of $\mathcal L(\design)$. 

The expected value of the random polynomial function $p(X)$ with respect of $\lambda$ is 
\begin{equation*}
\expect\left( p(X) \right) = \expect\left( g(X) \right) + \expect\left( r(X) \right) = 
\expect\left( g(X) \right) + \sum_{z\in \design} p(z) \expect\left(l_z(X) \right) 
\end{equation*}
by  linearity. In this paper we discuss classes of polynomials $p$ and design points $\design$ for which 
\begin{equation*}
\expect\left( p(X) \right) =  \sum_{z\in \design} p(z) \expect\left( l_z(X) \right) 
\end{equation*}
equivalently $\expect\left( g(X) \right)=0$.

In one dimension, 
the polynomial $f$ vanishing over $\design$ and of degree $n=|\design|$ forms
a Gr\"obner basis for $\mathcal I(\mathcal D)$.
Here $|A|$ indicates the number of elements of a set $A$. 
 Furthermore, $r$ satisfies three main properties:
\begin{enumerate}
\item  $r$ is a polynomial of degree less or equal to $n-1$,
\item $p(x)=g(x)+r(x)=q(x) f(x)+r(x)$ for a suitable $q\in \reals[x]$ and $g,f\in \langle f\rangle$, 
\item $r(x)=p(x)$ if $x$ is such that $f(x)=0$. Such an $x$ is in $\mathcal D$.
\end{enumerate}
In Section 2 we consider the algebra of orthogonal polynomials in one variable.


\section{Orthogonal polynomials and their algebra}
\label{section:OrthogonalPolynomialsOnTheLine}
In this section let $d=1$ and $\design$ be the zero set of a polynomial which is orthogonal to the constant functions with respect to $\lambda$.  We next recall the basics on orthogonal polynomials we use, see e.g. \cite{MR2061539}.

Let $I$ be a finite or infinite interval of $\reals$ and $\lambda$ a positive measure over $I$ such that all moments $\mu_j=\int_I x^j \, d\lambda(x)$, $j=0,1,\ldots$, exist and are finite. In particular, each polynomial function is square integrable on $I$ and the $L^2(\lambda)$ scalar product is defined by
\begin{equation*}
\scalarat \lambda {f}{g} = \int_I f(x)g(x)\,d\lambda(x)  
\end{equation*}

We consider only $\lambda$ whose related inner product is definite positive, i.e. $||f|| = \sqrt{ \scalarat \lambda {f}{f} }>0$ if $f\neq 0$. In this case there is a unique infinite sequence of monic orthogonal polynomials with respect to $\lambda$ and we denote them as
\begin{equation*}
\pi_0(x) =1,\, \pi_1(x) = x + \cdots, \, \pi_2(x) = x^2 + \cdots, \, \ldots
\end{equation*}

Furthermore we have $\pi_k\in \reals[x]_k$; $\pi_0,\ldots,\pi_k$ form a real vector space basis of $R[x]_{\leq k}$; $\pi_k$ is orthogonal to all polynomials of total degree smaller than $k$;  for $p\in \reals[x]$ and $n\in \mathbb Z_{\geq 0}$ there exists unique $c_n(p)\in \reals$, called 
$n$-th Fourier coefficient of $p$, such that $p(x) = \sum_{n=0}^{+\infty} c_n(p) \pi_n(x)$ and only a finite number of $c_n(p)$ are not zero.

Since the inner product satisfies the \emph{shift property}, i.e. 
\begin{equation*} \scalarat \lambda {xp(x)} {q(x)} = \scalarat \lambda {p(x)} {xq(x)} \text{ for } p,q\in\mathbb R[x]
\end{equation*}
then the corresponding orthogonal polynomial system satisfies a \emph{three-term recurrence} relationship.  More precisely, all orthogonal polynomial systems on the real line satisfy a three-term recurrence relationships. Conversely, Favard's theorem holds~\cite{MR1808151}.

\begin{theorem}[Favard's theorem] \label{Favard} Let $\gamma_n,\alpha_n,\beta_n$ be sequences of real numbers and for $n\geq 0$ let
  \begin{equation*}
    \pi_{n+1}(x)=(\gamma_nx-\alpha_n) \pi_n(x)-\beta_n \pi_{n-1}(x)
  \end{equation*}
be defined recurrently with $\pi_0(x)=1$, $\pi_{-1}(x)=0$. The $\pi_n(x)$, $n=0,1,\ldots$ form a system of orthogonal polynomials if and only if $\gamma_n\neq 0$, $\alpha_n\neq 0$ and $\alpha_n \gamma_n \gamma_{n-1}>0$ for all $n\geq 0$. If $\gamma_n=1$ for all $n$ then the system is of monic orthogonal polynomials.  \end{theorem}

In the monic case,
\begin{equation*}
\alpha_k=\displaystyle\frac{\langle x \pi_k,\pi_k \rangle}{\langle \pi_k,\pi_k \rangle} \quad \text{ and } \quad
\beta_k=\displaystyle\frac{\langle \pi_k,\pi_k \rangle}{\langle \pi_{k-1},\pi_{k-1} \rangle}
\end{equation*}
hold true and therefore the norm of $\pi_n$ is computed from the $\beta$'s as $\normof{\pi_n}^2=\beta_n\beta_{n-1}\ldots \beta_0$.

For orthonormal polynomials $\widetilde{\pi}_k=\pi_k/\normof{\pi_k}$ the {Christoffel-Darboux formul{\ae}} hold
\begin{equation} \label{CD-formula}
\begin{split} 
\sum_{k=0}^{n-1}  \widetilde \pi_k(x)  \widetilde \pi_k(t) &= \sqrt{\beta_{n}} \displaystyle\frac{ \widetilde \pi_{n} (x)  \widetilde\pi_{n-1}(t) - \widetilde\pi_{n-1}(x)  \widetilde \pi_{n}(t) }{x-t}  \\ 
\sum_{k=0}^{n-1}  \widetilde \pi_k(t)^2 &= \sqrt{\beta_{n}} \left( 
\widetilde \pi_{n}^\prime (t)  \widetilde\pi_{n-1}(t) - \widetilde\pi_{n-1}^\prime(t)  \widetilde \pi_{n}(t)  
 \right)   
 \end{split}
 \end{equation}
 
\begin{example}
Inner products of the Sobolev type, namely 
$\langle u,v \rangle_S=\langle u,v \rangle_{\lambda_0} + \langle u^\prime,v^\prime \rangle_{\lambda_1}
\cdots+\langle u^{(s)}, v^{(s)} \rangle_{\lambda_s}
$ where ${\lambda_i}$ are positive measures possibly having different support, do not satisfy the shift condition. Neither do the complex Hermitian inner products.
\end{example}

\begin{theorem}\label{main_theo_onedim} 
Let $\design=\left\{ x\in \reals: \pi_n(x)=0\right\}$ be the zero set of the $n$-th orthogonal polynomial with respect to the distribution $\lambda$ of the real random variable $X$. 
Consider the division of  $p\in \mathbb R[x]$ 
 by $\pi_n$ giving $p(x) = q(x)\pi_n(x)+r(x)$ as above. 
 Then there exist weights $\lambda_z$, $z \in \design$, such that the expected value of $p(X)$ is
\begin{equation*}
\expect\left( p(X) \right) = c_n(q) \normof{\pi_n}_\lambda^2 + \sum_{z \in \design} p(z) \lambda_z \end{equation*}
where $\lambda_z = \expect\left( l_z(X) \right)$ and $c_n(q)$ is the $n$-th Fourier coefficient of the polynomial $q$. 
\end{theorem}

\begin{remark}
 This theorem is a version of a well known result, see e.g. \cite[Sec. 1.4]{MR2061539}. We include the proof to underline a particular form of the error in the quadrature formula, to be used again in Theorem \ref{main_theo_onedim_general} and in Section 6.
Applying Theorem~\ref{main_theo_onedim} to $p=1$ we have that $1=\sum_{z \in \design} \lambda_z $. 
\end{remark}

\begin{proof} 
The set $\design$ contains $n$ distinct points. For a univariate polynomial $p$, we can write uniquely $p(x) = q(x) \pi_n(x)+r(x)$ with $\deg(r)<n$ and $\deg(q) = \operatorname{max}\{\deg(p)-n, 0\}$.  Furthermore, the indicator functions in the expression $r(x)=\sum_{z\in \design} p(z) l_{z}(x)$ are the Lagrange polynomials for $\design$: namely
\begin{equation*}
  l_{z}(x)=\prod_{w\in \design: w\neq z} \displaystyle\frac{x-w}{z-w}, \quad z \in \design. 
\end{equation*}
Hence we have
\begin{align*}
\expect\left( p(X) \right) &= 
\expect\left( q(X) \pi_n(X) \right) + \sum_{z\in \design} p(z) \expect\left( l_z(X) \right) 
\\ & = \sum_{k=0}^{+\infty} c_k(q) \expect\left( \pi_k(X)\pi_n(X)   \right) + \sum_{z\in \design} p(z)  \lambda_z 
\\  & =  c_n(q) \normof{\pi_n}_\lambda^2 + \sum_{z\in \design} p(z)  \lambda_z 
\end{align*}
as $\expect\left( \pi_k(X) \pi_n(X)\right) = \delta_{k,n}$ with $\delta_{k,n}=1$ if $k=n$ and $0$ otherwise.
\qed
\end{proof}

A particular case of Theorem~\ref{main_theo_onedim} occurs if $p$ has degree less than $2n$. In this case $q$ has degree at most $n-1$ and $c_n(q)=0$.  This shows that the quadrature rule with $n$ nodes given by the zeros of $\pi_n$ and weights $\{ \lambda_z \}_{z\in \design}$ is a Gaussian quadrature rule and it is exact for all polynomial functions of degree smaller or equal to $2n-1$. For notes on quadrature rules see for example \cite[Ch.~1]{MR2061539}. 

\begin{example}[Identification]
For $f$ polynomial of degree $N\leq 2n-1$ we can write  
$f(x) = \sum_{k=0}^N c_k(f) \pi_k(x)$. The constant term is given by 
\begin{equation*}  c_0(f) = \expect(f(X)) = \sum_{z\in \design} f(z) \lambda_z    \end{equation*}
and for all $i$ such that $N+i\leq 2n-1$
\begin{equation*} ||\pi_i||^2_\lambda c_i(f)
= \expect(f(X)\pi_i(X) ) =
\sum_{ z\in \design} f(z) \pi_i(z) \lambda_z   \end{equation*}
In particular, if $\deg f=n-1$ then all coefficients in the Fourier expansion of $f$ can be computed with an evaluation on $\design$. 

In general for  a polynomial of degree $N$ possibly larger than $2n-1$, Theorem~\ref{main_theo_onedim} gives  the Fourier expansion of its reminder by $\pi_n$, indeed 
\begin{multline*}
  \sum_{z\in \design} f(z) \pi_i(z) \lambda_z
 =
\sum_{z\in \design} \operatorname{NF} ( f  \pi_i  )(z) \lambda_z  
 \\ 
 = \expect(\operatorname{NF} (  f(X)\pi_i(X)  )) = ||\pi_i||^2_\lambda   c_i(\operatorname{NF}(f) )
\end{multline*}

\end{example} 

Theorem~\ref{main_theo_onedim_general} below generalises Theorem~\ref{main_theo_onedim} to a generic finite set of $n$ distinct points in $\reals$, say $\design$.  As above, the indicator function of $z\in \design$ is $ l_{z}(x)=\prod_{w\in \design: w\neq z} \displaystyle\frac{x-w}{z-w} $.  Let $g(x) = \prod_{z\in \design} (x-z)$ be the unique monic polynomial vanishing over $\design$ and of degree $n$.  Write a polynomial $p\in \reals[x]$ uniquely as $p(x)= q(x) g(x) + r(x)$ and consider the Fourier expansions of $q$ and $g$: $q(x) = \sum_{k=0}^{+\infty} c_k(q) \pi_k(x)$ and $g(x) = \sum_{k=0}^{n} c_k(g) \pi_k(x)$.

\begin{theorem} \label{main_theo_onedim_general}
With the above notation,
\begin{equation*}
  \expect_{\lambda}\left( p(X) \right) = \sum_{k=0}^{+\infty} c_k(q) c_k(g) \normof{\pi_k}_\lambda^2 + \sum_{z\in \design} p(z) \lambda_z.
\end{equation*}
\end{theorem}  
\begin{proof} From 
\begin{align*}
p(x) &= q(x)g(x)+r(x) \\
&=  \sum_{k=0}^{+\infty} c_k(q)\pi_k(x) \sum_{j=0}^{n} c_j(g)\pi_j(x)  +\sum_{z\in \design} p(z) l_z
 \end{align*}
%
we have
\begin{align*}
\expect\left( p(X) \right) &=  \sum_{k=0}^{+\infty} \sum_{j=0}^{n} c_k(q) c_j(g)  \expect\left( \pi_k(X) \pi_j(X)\right) 
 +\sum_{z\in \design} p(z) \lambda_z \\ &=  \sum_{k=0}^{n} c_k(q) c_k(g) || \pi_k ||_\lambda^2 + \sum_{z\in \design} p(z) \lambda_z  
 \end{align*}
 and this proves the theorem.\qed
\end{proof}

The condition in Theorem~\ref{main_theo_onedim_general} is linear in the Fourier coefficients of $q$, which is found easily from $f$ by polynomial division. 
The first $| \design |$ Fourier coefficients of $q$ appearing in the conditions of the theorem are determined by solving the system of linear equations
\begin{align} \label{Fourier_coeff_one_dim}
M \left[ c_k(q)\right]_{k=0,\ldots,| \design | -1} =  \left[ q(z)\right]_{k=0,\ldots,| \design | -1}
\end{align}
where $M = \left[ \pi_k(z) \right]_{z\in \design, k=0,\ldots,|\design | -1}$ is  the design/evaluation matrix for the first $| \design | $ orthogonal polynomials.

Theorem~\ref{main_theo_onedim_general} can be used in two ways at least. If $p$ is known, the condition in the theorem can be checked to verify if the expected value of $p$ can be determined by Gaussian quadrature rule with nodes $\design$ and weights 
\begin{equation*}
 \lambda_z = \expectof{\prod_{w\in \design: w\neq z} \displaystyle\frac{X-w}{z-w}} = \frac{\sum_0^{n-1} \alpha(z,k)\expectof{X^k}}{\sum_0^{n-1} \alpha(z,k)z^k}
\end{equation*}
for ${z\in \design}$, where $\alpha(z,k)$ is the $k$-th symmetric function of the polynomial $\pi_n(x)/(x-z)$. 
The Fourier coefficients of $g$ can be computed analogously to those of $q$ adapting Equation~(\ref{Fourier_coeff_one_dim}).  

If $p$ is an unknown polynomial and $p(x)=\sum_\alpha p_\alpha x^\alpha$ for a finite number of non-zero, unknown, real coefficients $p_\alpha$, Theorem~\ref{main_theo_onedim_general} characterizes all the polynomials for which the Gaussian quadrature rule is exact, namely $\expect_{\lambda}\left( p(X) \right) = \sum_{z\in \design} p(z) \lambda_z$.  Furthermore, the characterization is a linear expression in the unknown $p_\alpha$.  
This is because in Equation~(\ref{Fourier_coeff_one_dim}) the $q(z)$ are linear combinations of the coefficients of $p$.

In Section~\ref{Hermite:section} we shall specialise our study to Hermite polynomials, while in Section~\ref{section:HigherDimension} we  shall generalise Theorem~\ref{main_theo_onedim_general} to higher dimension. 
To conclude this section, we discuss the remainder $r$ vs the orthogonal projection.
\begin{remark} \label{proiezione_th}
Let $p(x)\in \reals[x]$ and write $p(x)=q(x) \pi_n(x) + r(x)$ where $r$ has degree less than $n$. Then
\begin{enumerate}
\item $q$ is the unique polynomial such that $p-q\pi_n$ is orthogonal to all $\pi_m$ with $ m \ge n$. This is a rephrasing of the characteristic property of the remainder: $r$ belongs to $\reals[x]_{\leq n-1}$ if, and only if, $r$ is orthogonal to all $\pi_m$ with $ m \ge n$. Should it exist two such $q$'s, then $(q_1-q_2)\pi_n$ would be in the same space, hence null.
\item If $\deg(p) = n$, then $r$ is the orthogonal projection of $p$ on $\reals[x]_{\leq n-1}$. In fact, $q$ is the leading coefficient of $p$, therefore $p - r$ is a multiple of $\pi_n$ and indeed orthogonal to $\reals[x]_{\leq n-1}$.
\item If $\deg(p) = N \ge n$, then the orthogonal projection of $p$ on $\reals[x]_{\leq n-1}$ differs from $r$, unless the projection of $q\pi_n$ is zero.
\end{enumerate}
 \end{remark}

\begin{example} Substituting the Fourier expansions of $q$ and $p$ in the division above, for $m\geq n$ we find that the $m$-th coefficient in the Fourier expansion of $p$ can be written as \begin{align*}
 \expect\left( p(X) \pi_m(X) \right) & =  \expect\left( q(X)\pi_n(X) \pi_m(X) \right) \\
\sum_{k=0}^{+\infty} c_k(p) \expect\left( \pi_k(X)\pi_m(X) \right) & =  \sum_{j=0}^{+\infty} c_j(q)\expect\left( \pi_j(X)\pi_n(X) \pi_m(X) \right) \\
 c_m(p) || \pi_m ||^2 &= \sum_{j=0}^{+\infty} c_j(q)\expect\left( \pi_j(X)\pi_n(X) \pi_m(X) \right)
\end{align*}
For Hermite polynomials it can be simplified by e.g. using the product formula in Theorem~\ref{productHermite} of Section~\ref{Hermite:section}.  
\end{example}


\section{Hermite polynomials}\label{Hermite:section}
There is another way to look at the algebra of orthogonal polynomials that we discuss here in the case of Hermite polynomials. 
The reference measure $\lambda$ is the normal distribution and $d\lambda(x) = w(x) \,dx$, with $w(x) = e^{-x^2/2}/\sqrt{2\pi} $, $x\in \reals$.

\subsection{Stein-Markov operators for standard normal distribution}
For a real valued, differentiable function $f$, define 
\begin{equation*}
  \delta f(x)= x f(x) - \displaystyle\frac{d}{dx} f(x) = -e^{x^2/2} \displaystyle\frac{d}{dx} \left( f(x) e^{-x^2/2} \right), 
\end{equation*}
$d^n=\displaystyle\frac{d^n}{dx^n}$, and consider $Z\sim \lambda$.
The following identity holds
\begin{equation} \label{malliavinderivative}
 \expect \left( \phi(Z) \ \delta^n \psi(Z) \right) =\expect \left( d^n  \phi(Z) \ \psi(Z) \right) 
\end{equation}
if $\phi,\psi$ are such that $\lim_{x\rightarrow \pm \infty}  \phi(x)\psi(x)e^{-x^2/2}=0$ and are square integrable, see \cite[Ch. V Lemma 1.3.2 and Proposition 2.2.3]{MR1335234}). Polynomials satisfy these conditions and $\delta$ is also called the Stein-Markov operator for the standard normal distribution. It is a shift-equivariant linear operator on $\reals[x]$, namely $\delta f(x+a)=(x+a)f(x+a) - f'(x+a) =(\delta f)(x+a)$ holds for $a\in \reals$ and the constant one is mapped into $x$.

The $n$-th Hermite polynomial can be defined as $H_n(x)=\delta^n 1$. 
Direct computation using $\delta$ proves the following well-known facts:
\begin{enumerate} 
\item The first Hermite polynomials are
\begin{align*} 
&H_0=1 \\ &H_1(x)=x \\ &H_2(x)=x^2-1 \\ &H_3(x)=x^3-3x \\ &H_4(x) = x^4-6x^2+3 
\\ &H_5(x) = x^5-10x^3+ 15 x 
\end{align*}
\item $H_n(x) = (-1)^n e^{x^2/2} d^n ( e^{-x^2/2} ) $ (Rodrigues' formula) 
\item $d\delta -\delta d $ is the identity operator. From this the relationships $dH_n=n H_{n-1}$, 
$ d^m H_n  = \displaystyle \frac{n!}{m!} H_{n-m}$ for $m\leq n $
and the three-term recurrence relationship 
\begin{eqnarray} \label{recurrence_relationship}
H_{n+1}=xH_n-nH_{n-1}
\end{eqnarray}
 are deduced.
\item Hermite polynomials are orthogonal with respect to the standard normal distribution $\lambda$. Indeed from Equation~(\ref{malliavinderivative}) we have $\expect\left( H_n(Z) H_m(Z) \right)= n! \delta_{n,m}$ where $\delta_{n,m}=0$ if $n\neq m$ and $\delta_{n,m}=1$ if $n=m$.
\end{enumerate}

We already mentioned that $\{ H_n(x): n\leq d\}$ spans $\reals[x]_{\leq d}$ and that $H_n$ is orthogonal to any polynomial of degree different from $n$. The ring structure of the space generated by the Hermite polynomials is described in Theorem \ref{productHermite}.   

\begin{theorem} \label{productHermite}
The Fourier expansion of the  product $H_kH_n$ is 
\begin{equation*}  H_k H_n = H_{n+k}  + \sum_{i=1}^{n\wedge  k}  \binom{n}{i} \binom{k}{i} i! H_{n+k-2i}  \end{equation*}
 \end{theorem} 
 \begin{proof}
Note that  $\langle \phi,\psi\rangle = \expect(\phi(Z) \psi(Z))$ is a scalar product on the obvious space and let $n\leq k$ with $Z\sim \mathcal N(0,1)$ and $\psi,\phi$ square integrable functions for which identity~(\ref{malliavinderivative}) holds. Then
\begin{align*}
 \langle {H_k H_n}, \psi \rangle 
& =  \langle \delta^n 1, H_k \psi \rangle \\ &=  \langle 1,  d^n( H_k \psi ) \rangle = \sum_{i=0}^n \langle 1, \binom{n}{i} d^i H_k \ d^{n-i} \psi \rangle \\
& = \langle 1, H_k d^n \psi \rangle +  \sum_{i=1}^n \langle 1, \binom{n}{i} d^i H_k \ d^{n-i} \psi \rangle \\
= \langle H_{n+k} , \psi  \rangle &+ \sum_{i=1}^n \binom{n}{i} k (k-1)\dots(k-i+1) \langle H_{n+k-2i} , \psi \rangle \\
& = \langle H_{n+k} , \psi  \rangle   + \langle  \sum_{i=1}^n \binom{n}{i} \binom{k}{i} i! H_{n+k-2i}  , \psi \rangle
 \end{align*} 
 \qed
\end{proof}
 
\begin{example}[Aliasing]  \label{confounding_Example_Hermite_one_dim}
As an application of Theorem~\ref{productHermite}, observe that the three-term recurrence relation for Hermite polynomials, Equation~(\ref{recurrence_relationship}) 
\begin{eqnarray*} 
H_{n+1}=xH_n-nH_{n-1}
\end{eqnarray*}
evaluated on the zeros of $H_n(x)$, say $\design_n$,  becomes
$H_{n+1}(x) \equiv -n H_{n-1}(x)$ where $\equiv$ indicates that equality holds for $x\in \design_n$. 
In general let $H_{n+k}  \equiv \sum_{j=0}^{n-1} h_j^{n+k} H_j $ be the Fourier expansion of the normal form of $H_{n+k}$ at $\design_n$, 
where we simplified the notation for the Fourier coefficients. 
Substitution in the product formula in Theorem~\ref{productHermite} gives the formula to write $h_j^{n+k}$ in terms of Fourier coefficients of lower order Hermite polynomials: 
\begin{eqnarray} 
\operatorname{NF}(H_{n+k}) 
& \equiv  -\sum_{i=1}^{n \wedge k}  {n \choose i} {k \choose i}  i! \operatorname{NF}( H_{n+k-2i} ) \label{peccati} \\ 
& \equiv - \sum_{i=1}^{n \wedge k}  {n \choose i} {k \choose i}  i! \sum_{j=0}^{n-1} h_j^{n+k-2i} H_j \nonumber
\end{eqnarray}
Equating coefficients gives a closed formula
\begin{equation*}
 h_j^{n+k} = -\sum_{i=1}^{n \wedge k}  {n \choose i} {k \choose i}  i! h_j^{n+k-2i} 
\end{equation*}

In Table~\ref{confouding_Hermite} the normal form of $H_{k+n}$ with respect to $H_n$ is written in terms of Hermite polynomials of degree smaller than $n$. For example, $H_{n+3}(x) = - n (n-1)(n-2) H_{n-3}(x) +3 n H_{n-1}(x)$ for those values of $x$ such that $H_n(x)=0$.
\begin{table}[h] 
\begin{equation*}
\begin{array}{l|l}
k & H_{n+k} \equiv \\ \hline
  1 & -n H_{n-1} \\ 
  2 & -n (n-1) H_{n-2} \\
  3 & - n (n-1)(n-2) H_{n-3} +3 n H_{n-1} \\ 
  4 & - n (n-1)(n-2)(n-3) H_{n-4}+8n(n-1)H_{n-2} \\ 
  5 & - \frac{n!}{(n-5)!} H_{n-5}+ 5n H_{n-1}+15 n(n-1)(n-2)H_{n-3} \\ 
  6 & - \frac{n!}{(n-6)!} H_{n-6}+24n(n-1)(n-2)(n-3)H_{n-4} \\ & \multicolumn{1}{r}{+10n(n-1)(2n-5)H_{n-2}}
\end{array} 
\end{equation*}
\caption{Aliasing of $H_{n+k}$, $k=1,\dots,6$ over $\design = \{H_n(x)=0\}$}\label{confouding_Hermite}
\end{table}
\end{example} 

\begin{example}
Observe that if $f$ has degree $n+1$ equivalently $k=1$ then  
\begin{multline*}
f = \sum_{i=0}^{n-1} c_i(f) H_i +\sout{c_n(f) H_n} + c_{n+1}(f)H_{n+1}
 \\ \equiv \sum_{i=0}^{n-2} c_i(f) H_i + \left(c_{n-1} (f)-n c_{n+1}(f) \right) H_{n-1} 
\end{multline*}
and all coefficients up to degree $n-2$ are ``clean''.

 \end{example} 

We give another proof of Theorem~\ref{main_theo_onedim} for Hermite polynomials.
\begin{corollary}
Let $\design_n=\{x: H_n(x)=0 \}$ and $p\in \reals[x]$. 
Let $p(x)=q(x) H_n(x)+r(x)$ with the degree of $r$ smaller than $n$ 
and let $Z\sim\mathcal N(0,1)$. Then
\begin{equation*}
\expect\left( p(Z)\right) = \sum_{z\in \design_n} p(z) \lambda_z \text{ if and only if } \expect\left( d^n q (Z)\right)=0
\end{equation*}
with $\lambda_z=\expect\left( l_z(Z) \right)$ and $l_{z}(x)=\prod_{w\in \design: w\neq z} \displaystyle\frac{x-w}{z-w} $, $z\in \design_n$.
\end{corollary}
\begin{proof}
From Equation~(\ref{malliavinderivative}) we have
\begin{equation*}
\expect\left( q(Z) H_n(Z) \right) = \expect\left( q(Z) \delta^n 1 \right) = \expect\left( d^n q(Z) \right)
\end{equation*}
Now by the same steps followed in the proof of Theorem~\ref{main_theo_onedim} we conclude that 
\begin{equation*}
\expect\left( p(Z) \right) = \expect\left( d^n q(Z) \right) +\sum_{z\in \design_n} p(z) \lambda_z
\end{equation*} 
\qed   \end{proof}

\subsection{Algebraic characterisation of the weights}

Theorem~\ref{weighing_theorem} gives two polynomial equations whose zeros are the design points and the weights. 
This is a particular case of Formula~(1.17) in~\cite{StroudSecrets}.
We provide the proof to highlight the algebraic nature of the result and of its proof.
 
 \begin{theorem} \label{weighing_theorem}
 Let $ \design=\{x: H_n(x)=0\}  $.  
\begin{enumerate}
\item There exists only one polynomial $\lambda$ of degree $n-1$ such that $\lambda(x) = \lambda_x$ for all $x \in \design$, 
 \item furthermore $ \lambda_x= \frac{(n-1)!}{n} H_{n-1}^{-2}(x)  $. Equivalently
\item  the polynomial $\lambda$ satisfies 
\begin{equation*}   
\left\{  \begin{array}{ll} 
    H_n(x)=0      \\ 
     \lambda(x) H_{n-1}^2(x)= \displaystyle\frac{(n-1) !}n
   \end{array}   \right.  \end{equation*}
   \end{enumerate} 
\end{theorem}
\begin{proof} 
\begin{enumerate}
\item The univariate polynomial $\lambda$ is the interpolation polynomial of the values $\lambda_z$'s at the $n$ distinct points in $\design_n$ and hence it exists, unique of degree $n-1$. 
\item Observe that for Hermite polynomials  $\alpha_n=0$, $\beta_n=n$, $ \widetilde H_n(x) = {H_n(x)}/ {\sqrt{n!}} $ and $\widetilde H_n^\prime(x) = \sqrt n \widetilde H_{n-1}(x) $. Substitution in the Christoffel-Darboux formul{\ae} and evaluation at $\design_n=\{x_1,\ldots,x_n\}$ give
\begin{equation} \label{CD-formula-on-D}
\begin{split} \sum_{k=0}^{n-1} \widetilde H_k(x_i) \widetilde H_k(x_j) =0  \text{ if } i\neq j &\\
\sum_{k=0}^{n-1} \widetilde H_k(x_i)^2  = n \widetilde H_{n-1}(x_i)^2  
\end{split}
\end{equation} 
In matrix form Equations~(\ref{CD-formula-on-D}) becomes
\begin{align*} 
\mathbb H_n \mathbb H_n^t & = n \, \diag( \widetilde H_{n-1}(x_i)^2: i=1,\ldots,n) 
\end{align*} 
where $\mathbb H_n$ is the square matrix  $\mathbb H_n=\left[ \widetilde H_j(x_i) \right]_{i=1,\ldots,n; j=0,\ldots,n-1}$ and $\diag$ indicates a diagonal matrix.
Observe that $\mathbb H_n$ is invertible and
\begin{align*} 
 \mathbb H_n^{-1}          &=   \mathbb H_n^t n^{-1} \diag( \widetilde H_{n-1}^{-2}(x_i): i=1,\ldots,n)
\end{align*} 

Now, let $f$ be a polynomial of degree at most $n-1$, that is a typical remainder by division for $H_n$, then $f(x) = \sum_{j=0}^{n-1} c_j \widetilde H_j(x)$.
Write $\underline f=\mathbb H_n \underline c$ where  $\underline f=[f(x_i)]_{i=1,\ldots,n}$ and $\underline c=[c_j]_j$.
Furthermore note that 
\begin{align}
\underline c
&=\mathbb H_n^{-1} \underline f  = \mathbb H_n^t n^{-1} \diag( \widetilde H_{n-1}^{-2}(x_i): i=1,\ldots,n) \underline f \nonumber \\ &= \mathbb H_n^t n^{-1} \diag( \widetilde H_{n-1}^{-2}(x_i) f(x_i): i=1,\ldots,n)  \nonumber \\
c_j &= \frac1n \sum_{i=1}^n \widetilde H_j(x_i) f(x_i) \widetilde H_{n-1}^{-2}(x_i)
\label{chaosexpcoeffs} 
\end{align}
Apply this to  the $k$-th Lagrange polynomial,   $f(x)=l_k(x)$, whose Fourier expansion is $f(x)= \sum_{j=0}^{n-1} c_{kj} \widetilde H_j(x) $. 
Using $l_k(x_i)=\delta_{ik}$ in Equation~(\ref{chaosexpcoeffs}), obtain 
\begin{align} \label{coefflagrange}
c_{kj} = \frac1n \widetilde H_j(x_k) \widetilde H_{n-1}^{-2}(x_k) 
\end{align}
The expected value of $l_k(Z)$ is 
\begin{equation*}
\lambda_k=\expect\left( l_k(Z) \right) = \sum_{j=0}^{n-1} c_{kj} \expect\left( \widetilde H_j(x) \right)
=c_{k0} 
\end{equation*}
Substitution in Equation~(\ref{coefflagrange}) for $j=0$ gives
\begin{equation*}
  \lambda_k =  \frac1n  \widetilde H_{n-1}^{-2}(x_k) = 
\frac{(n-1)!}n   H_{n-1}^{-2}(x_k)  
\end{equation*}
This holds for all $k=1,\ldots,n$. 
\item The system of equations is a rewriting of the previous parts of the theorem because the first equation $H_n(x)=0$ states that only values of $x\in\design_n$ are to be considered  and the second equation is what we have just proven. 
\end{enumerate}
\qed\end{proof}

Item~2 in Theorem \ref{weighing_theorem} states that the weights are strictly positive.   
Theorem~\ref{main_theo_onedim} applied to the constant polynomial $p(x)=1$ shows that they sum to one. In other words, the mapping that associates
$z$ to $\lambda_z$,  $z\in \design_n$, is a discrete probability density. 
Theorem~\ref{main_theo_onedim} states that the expected value of the polynomial functions of $Z~\sim \mathcal N(0,1)$ 
for which $c_n(q)=0$, is equal to the expected value of a discrete 
random variables $X$ given by $ \operatorname{P_n} \left( X=x_k \right) =  \expect \left( l_k(Z) \right)=  \lambda_k$, $k=1,\ldots,n$ 
\begin{equation*} 
\expect \left( p(Z) \right)  = \sum_{k=1}^n p(x_k)  \lambda_k = \operatorname{\mathbb E_n} \left( p(X) \right) 
\end{equation*}

\begin{example}
For $n=3$ the polynomial $\lambda$ in Theorem~\ref{weighing_theorem} can be determined by-hand. For larger values of $n$ an algorithm is provided in Section~\ref{weighing-code}.
The polynomial system to be considered is 
\begin{align*}
0 &=H_3(x)=x^3-3x \\
2/3 &= \lambda(x) H^2_2=(\theta_0+\theta_1x+\theta_2 x^2) (x^2-1)^2
\end{align*}
where $\lambda(x) = \theta_0+\theta_1x+\theta_2 x^2$. 
The degree of $ \lambda(x) H_2^2$ is reduced to $2$ by using $x^3=3x$ 
\begin{equation} \label{weights-example-3}
 2/3 = \lambda(x) H^2_2 = \theta_0+\theta_1 4 x + (\theta_0+4 \theta_2) x^2
\end{equation}
Coefficients in Equation~(\ref{weights-example-3}) are equated to give $\lambda(x)=\frac23-\frac{x^2}6$.
\end{example} 

In some situations, e.g. the design of an experimental plan or of a Gaussian quadrature rule, the exact computation of the weights might not be necessary and $\lambda(x)$ is all we need. When the explicit values of the weights are required, the computation has to be done outside a symbolic computation setting as 
 we need to solve, e.g., $H_3(x)=0$ to get $\design_3=\{ -\sqrt{3},  0,\sqrt{3} \}$
and evaluate $\lambda(x)$  to find $\lambda_{-\sqrt 3}= \lambda(-\sqrt 3) =\frac16=\lambda_{\sqrt 3}$ and 
$\lambda_0= \lambda(0)=\frac23$. 

\begin{example}  \rm
Let a positive integer $N$ be given and let $k$ and $n$ be positive integers such that $kn<2N$ then 
\[
\expect \left(H_n(Z)^k \right) = \sum_{z\in \mathcal D_N } \displaystyle\frac{(n-1)!}{n}  \displaystyle\frac{H_n(z)^k }{H^2_{N-1}(z)}
\]
by Theorem~\ref{weighing_theorem}. The issue is then the evaluation of $H_{N-1}$ and $H_n$ at the zeros of $H_N$, for which the recurrence relationship can be used when the values are not tabulated.
\end{example} 

\subsection{Code for the weighing polynomial}\label{weighing-code}

The polynomial $\lambda(x)$ in Theorem~\ref{weighing_theorem} is called the \emph{weighing polynomial}.  Table~\ref{weighing-code-table} gives a code written in the specialised software for symbolic computation called CoCoA~\cite{CocoaSystem}  to compute the Fourier expansion of $\lambda(x)$ based on  Theorem~\ref{weighing_theorem}.

Line 1 specifies the number of nodes $N$.  Line 2 establishes that the working environment is a polynomial ring whose variables are the first $(N-1)$-Hermite polynomials plus an extra variable $w$ which encodes the weighing polynomial; here it is convenient to work with a elimination term-ordering of $w$, called {\tt Elim(w)}, so that the variable $w$ will appear as least as possible.  Lines 3, 4, 5 construct Hermite polynomials up-to-order $N$ by using the recurrence relationships~(\ref{recurrence_relationship}). Specifically they provide the expansion of $H_j$ over $H_k$ with $k<j$ for $k=0,\ldots,N-1$.  Line~6 states that $H_N= H_1 H_{N-1}-(N-1) H_{N-2}=0$, `giving' the nodes of the quadrature.  Line~7 is the polynomial in the second equation in the system in Item~3 of Theorem~\ref{weighing_theorem} and `gives' the weights.  
There are $N$ equations which are collected in an ideal whose Gr\"obner basis is computed in Line~8.  In our application it is interesting that the Gr\"obner bases contains a polynomial in which $w$ appears alone as a term of degree one. 
This element of the Gr\"obner basis relates explicitly to the desired weighing polynomial $w$ to the first $(n-1)$-Hermite polynomials.
\begin{center} 
\begin{table}[h] 
\begin{verbatim} 
L1  N:=4;                                  
L2  Use R::=Q[w,h[1..(N-1)]], Elim(w);     
L3  Eqs:=[h[2]-h[1]*h[1]+1];              
L4  For I:=3 To N-1 Do 
L5   Append(Eqs,h[I]-h[1]*h[I-1]+(I-1)*h[I-2]) EndFor;    
L6  Append(Eqs,h[1]*h[N-1]-(N-1)*h[N-2]);
L7  Append(Eqs,N*w*h[N-1]^2-Fact(N-1));    
L8  J:=Ideal(Eqs); GB_J:=GBasis(J); Last(GB_J); 
L9  3w + 1/4h[2] - 5/4                   
\end{verbatim}  
  \caption{Computation of the Fourier expansion of the weighing polynomial using Theorem~\ref{weighing_theorem}}
  \label{weighing-code-table}
 \end{table}
\end{center}  
Line~9 in Table~\ref{weighing-code-table} gives the polynomial obtained for $N=4$, as set in Line~1, namely
\begin{equation*}
  \lambda(x)=\left( - \frac14 H_2(x)+\frac54 \right) \frac13 =\frac{6-x^2}{12}
\end{equation*}
The nodes are  $\pm\sqrt{3\pm\sqrt{6}} $ and  the values of the weights are $\frac{3\pm \sqrt{6}}{12}$, showing that both nodes and weights are algebraic numbers but not rational numbers.
On a Mac OS X with an Intel Core 2 Duo processor (at 2.4  GHz) using CoCoA (release 4.7)
the result is obtained
for $N=10$ in Cpu time = 0.08, User time = 0;
for $N=20$ in Cpu time = 38.40, User time = 38;
for  $N=25$ in Cpu time = 141.28, User time = 142
and for $N=30$ in Cpu time = 5132.71, User time = 5186.
Observe that this computations can be done once for all and the results stored.
Observe furthermore that in Line~8
the \cocoa\ commands \texttt{GB{\_}J:=GBasis(J); Last(GB\_J);} 
could be substituted by  \texttt{NF(w,J)} . This
does not improve on computational cost as NF requires the computation of a Gr\"obner basis and a reduction.
As a minor point we observe that the symbol w would not appear in Line~9.


\section{Fractional design} \label{fraction_section}

In this section we return to the case of general orthogonal polynomials, $\{ \pi_n\}_n$, and positive measure, $d\lambda$. 
We assume that the nodes are a proper subset $\mathcal F$ of $\design_n=\{ x\in \mathbb R: \pi_n(x)=0 \}$ with  $m$ points, $0 < m <n$.
We work within two different settings, in one the ambient design $\design_n$ is considered while in the other one it is not.  

Consider the indicator function of $\mathcal F$ as subset of $\design_n$, namely 
$1_\mathcal F(x)=1$ if $x\in \mathcal F$ and $0$ if $x\in \design_n\setminus \mathcal F$. 
It can be represented by a polynomial of degree $n$ because it is a function defined over $\design_n$
\cites{MR2290010,MR2332740}. 
Let $p$ be a polynomial of degree at most $n-1$ so that the product $p(x)1_\mathcal F(x)$ is a polynomial of degree at most $2n-1$.
Then from Theorem~\ref{main_theo_onedim} we have 
\begin{multline*} 
  \expect((p 1_{\mathcal F})(X)) 
  = \sum_{z\in \mathcal F} p(z) \lambda_z 
  = \operatorname{{\mathbb E}_n} \left( p(Y) 1_{\mathcal F}(Y) \right)  
 \\ =   \operatorname{{\mathbb E}_n} \left( p(Y) | Y\in \mathcal F \right) 
\operatorname{P_n}(Y\in \mathcal F)   
 \end{multline*}
where $X$ is a random variable with probability law $\lambda$ and $Y$ is a discrete random variable taking value $z\in\mathcal F$ with probability  $\operatorname{P_n}(Y=z)= \lambda_z$.
The first equality follows from the fact that $p(x)1_\mathcal F(x)$ is zero for  $x\in \design\setminus \mathcal F$ and the last equality from the definition of conditional expectation. 

Another approach is to consider the polynomial whose zeros are the elements of $\mathcal F$, say $\omega_\mathcal F(x) = \underset{z\in \mathcal F}{\prod} (x-z)$. 
Now consider the Lagrange polynomials for $\mathcal F$, namely 
 $l_z^\mathcal F(x) =  \underset{\underset{w\in \mathcal F}{w\neq z}}{\prod} \displaystyle\frac{x-w}{z-w}$ for $z\in \mathcal F$.
 
 \begin{lemma} 
Let $\mathcal F\subset \design_n$.  The Lagrange polynomial for $z\in \mathcal F$ is the remainder
of the Lagrange polynomial for $z\in \design_n$ with respect to $\omega_\mathcal F(x)$, namely 
  \begin{equation*}
  l_z^\mathcal F(x)= \operatorname{NF} \left(l_z(x), \langle \omega_\mathcal F(x) \rangle \right)\end{equation*} 
 \end{lemma} 
 \begin{proof} 
 There exists unique $\operatorname{NF}(l_{z})(x)$, polynomial of degree small than $m$, such that
 \begin{equation*}
 l_{z}(x) = q(x) \omega_\mathcal F(x)+\operatorname{NF}(l_{z})(x)
 \end{equation*}
Furthermore, for $a\in \mathcal F$ we have $ l_{z}(a)=\operatorname{NF}(l_{z})(a)=\delta_{z,a}=  l_z^\mathcal F(a)$.
The two polynomials $  l_a^\mathcal F(x)$ and $\operatorname{NF}(l_{z})(x)$ have degree smaller than $m$ and coincide on $m$ points, by interpolation they must be equal.
  \qed   \end{proof}
 
  For a polynomial $p$ of degree $N$, write $p(x) = q(x) \omega_\mathcal F(x) + r(x)$ with $f(z)=r(z)$ if $z\in \mathcal F$ and $ r(x) = \sum_{z\in \mathcal F} p(z) l_z^\mathcal F(x) $.  Let $q(x) = \sum_{j=0}^{N-m} b_j \pi_j(x)$ and $\omega_\mathcal F(x)= \sum_{i=0}^m c_i \pi_i(x)$ as $\omega_\mathcal F$ has degree $m$. Then  
\begin{multline*}
\expect\left(p(X)\right) 
 = \expect\left( \sum_{j=0}^{N-m} b_j \pi_j(X) \sum_{i=0}^m c_i \pi_i(X) \right) + \expect\left(r(X)\right) = \\ b_0 c_0 || \pi_0 ||_\lambda^2 + b_1 c_1 || \pi_1 ||_\lambda^2 + \cdots \\ +  b_{(N-m) \wedge m} c_{(N-m) \wedge m} || \pi_{(N-m) \wedge m}  ||_\lambda^2
 + \sum_{z\in \mathcal F} p(z) \lambda_z^\mathcal F 
\end{multline*} 
where $\lambda_z^\mathcal F = \expect\left( \operatorname{NF}(l_z(X), \langle \omega_\mathcal F(X) \rangle \right)$, $z\in \mathcal F$. 

Note that the error of the Gaussian quadrature rule, 
\begin{multline*}
  b_0 c_0 || \pi_0 ||_\lambda^2 + b_1 c_1 || \pi_1 ||_\lambda^2 + \cdots + \\ b_{(N-m) \wedge m} c_{(N-m) \wedge m} || \pi_{(N-m) \wedge m}  ||_\lambda^2
\end{multline*}
is linear in the Fourier coefficients $b_j$, and also in the Fourier coefficients $c_j$ of the node polynomial. 
This is generalised in Section~\ref{section:HigherDimension}. 
If the fraction $\mathcal F$ coincides with the ambient design $\design_n$ and hence contains $n$ points 
and if $p$ is a polynomial of degree at most $2n-1$, 
then we obtain the well known result of zero error because $(N-n) \wedge n\leq n-1$ and
the only non-zero Fourier coefficient of the node polynomial $\pi_n$ is of order $n$.  
In general one should try to determine pairs of $\mathcal F$ and sets of polynomials for which 
 the absolute value of the errors is minimal.

\section{Higher dimension:  zero set of orthogonal polynomials as  design support}\label{section:HigherDimension}

In this section we return to the higher dimensional set-up of Section~\ref{section:generalframework} but we restrict ourselves to consider the product measure $\lambda^d=\times_{i=1}^d \lambda$ and $X_1,\ldots,X_d$ independent random variables each one of which is distributed according to the probability law $\lambda$. As design we take a product grid of zeros of orthogonal polynomials with respect to $\lambda$, more precisely our design points or interpolation nodes are
\begin{multline*}
\design_{n_1,\ldots,n_d} = \\ \left\{ x \in \reals^d: \pi_{n_1}(x_1)=\pi_{n_2}(x_2)=\ldots=\pi_{n_d}(x_d) =0 \right\}
\end{multline*}
where $\pi_{n_k}$ is the orthogonal polynomial with respect to $\lambda$ of degree $n_k$.
 
The Lagrange polynomial of the point $y=(y_1,\ldots,y_d)\in \design_{n_1,\ldots,n_d}$ is defined as 
$ l_y(x_1,\ldots,x_d)= \prod_{k=1}^d l_{y_k}^{n_k}(x_k) $, the apex $^{n_k}$ indicates that $l_{y_k}^{n_k}(x_k)$ is the univariate Lagrange polynomial for $y_k \in \{x_k: \pi_{n_k}(x_k)=0\} = \design_{n_k} \subset \reals$. 

The $\operatorname{Span}\left( l_y: y\in \design_{n_1,\ldots,n_d} \right)$ is equal to the linear space generated by the monomials whose exponents lie on the integer grid $\{ 0,\ldots,n_1-1 \}\times \ldots \times \{ 0,\ldots,n_d-1 \}$. 
Any polynomial $f\in \reals[x]$ can be written  as 
\begin{equation*}f(x_1,\ldots,x_d) = \sum_{k=1}^d q_k(x_1,\ldots,x_d) \pi_{n_k}(x_k) +r(x_1,\ldots,x_d)\end{equation*} 
where $r$ is unique, its degree in the variable $x_k$ is smaller than $n_k$, for  $k=1,\ldots,d$, and belongs to that $\operatorname{Span}$. 

The coefficients of the  Fourier expansion of $q_k$ with respect to the variable $x_k$
are functions of $x_1,\ldots,x_d$ but not of $x_k$. 
Let $x_{-k}$ denote the $(d-1)$-dimensional vector obtained from $(x_1,\ldots,x_d)$ removing the $k$-th component and write
\begin{multline*}f(x_1,\ldots,x_d) = \\ \sum_{k=1}^d 
\left( \sum_{j=0}^{+\infty} c_j(q_k)(x_{-k}) \pi_j(x_k) \right)
 \pi_{n_k}(x_k) +r(x_1,\ldots,x_d)\end{multline*} 
Only a finite number of $c_j(q_k)(x_{-k})$ are not zero.
 
From the independence of $X_1,\ldots,X_n$, the expected value  of the Lagrange polynomial $ l_y$ is
\begin{equation*} 
\expect_{ \lambda^d } \left( l_y(X_1,\ldots,X_d) \right) = 
\prod_{k=1}^d \expect_{ \lambda } \left( l_{y_k}^{n_k}(X_k) \right) = 
\prod_{k=1}^d \lambda_k^{n_k}
\end{equation*}
where $\lambda_k^{n_k} = \expect \left( l_{y_k}^{n_k}(X_k) \right) $ is the expected value of a univariate random Lagrange polynomial as in the previous sections.

\begin{theorem}\label{higherdim:theorem} 
It holds
\begin{multline*}
\expect_{\lambda^d} \left( f(X_1,\ldots,X_d)\right) = \\
\sum_{k=1}^d \expect_{\lambda^{d-1}} (c_k(q_k)(X_{-k}) )|| \pi_k||_\lambda^2
+ \\
\sum_{ (x_1,\ldots, x_n) \in \design_{n_1\ldots n_d}}
f(x_1,\ldots,x_d) \lambda_{x_1}^{n_1} \ldots \lambda_{x_d}^{n_d} 
\end{multline*} 
\end{theorem} 

\begin{proof} The proof is very similar to that of Theorem~\ref{main_theo_onedim} and we do it for $d=2$ only. 
In a simpler notation the design is the $n\times m$ grid given by 
$  \design_{nm} =\{(x,y): \pi_n(x)=0=\pi_m(y) \}$ and 
$X$ and $Y$ are independent random variables distributed according to $\lambda$. 
The polynomial $f$ is decomposed as 
\begin{multline*}
f(x,y) = \\ q_1(x,y) \pi_n(x) + q_2(x,y)\pi_n(y) + \sum_{(a,b)\in \design_{n,m} } f(a,b) \  l_a^n(x) l_b^m(y) = \\ \sum_{j=0}^{+\infty} c_j(q_1)(y) \pi_j(x) \ \pi_n(x) + \sum_{j=0}^{+\infty} c_j(q_2)(x) \pi_j(y) \  \pi_n(y) 
  + \\ \sum_{(a,b)\in \design_{n,m} } f(a,b) \ l_a^n(x) l_b^m(y)
\end{multline*}
Taking expectation, using independence of $X$ and $Y$ and orthogonality of the $\pi_i$, we have
\begin{multline*}
\expect_{\lambda^2}\left( f(X,Y) \right) 
 = \\
 \expect_{\lambda}\left( c_n(q_1)(Y) \right) || \pi_n ||_\lambda^2 
 +
 \expect_{\lambda}\left( c_m(q_2)(X) \right) || \pi_m ||_\lambda^2
 + \\ \sum_{(a,b)\in \design_{n,m} } f(a,b) \lambda_a^n \lambda_b^m
\end{multline*}\qed
\end{proof}

Note in the proof above that a sufficient condition for $\expect_{\lambda}\left( c_n(q_1)(Y) \right)$ 
being zero is that $f$ has degree in $x$ smaller then $2n-1$, similarly for $\expect_{\lambda}\left( c_m(q_2)(X) \right)$. 
We retrieve the well-known results that if for each $i$ the degree in $x_i$ of $f$ is smaller than $2 n_i-1$, then  
\begin{equation*}
\expect_{\lambda^d} \left( f(X_1,\ldots,X_d)\right) =
\sum_{ (x_1,\ldots, x_n) \in \design_{n_1\ldots n_d}}
f(x_1,\ldots,x_d) \lambda_{x_1}^{n_1} \ldots \lambda_{x_d}^{n_d} 
\end{equation*} 

In the Gaussian set-up, by Theorem~\ref{weighing_theorem} applied to each variable, weights and nodes satisfy the  polynomial system
\begin{equation} \label{finalHermite}  
\left\{ \begin{array}{ll}
   H_{n_1}(x_1) & = 0  \\
   \lambda_1(x_1) H_{n_1-1}(x_1)^2 & = \displaystyle\frac{(n_1-1)!}{n_1} \\
& \vdots \\
   H_{n_d}(x_d) & = 0 \\
   \lambda_d(x_d) H_{n_d-1}(x_d)^2 & = \displaystyle\frac{(n_d-1)!}{n_d} 
    \end{array}   \right.  
\end{equation}

For the grid set-up of this section and for the Gaussian case, in analogy to Example~\ref{confounding_Example_Hermite_one_dim} 
some Fourier coefficients of polynomials of low enough degree can be determined exactly from the values of the polynomials on the grid points as shown in Example~\ref{confounding_Example_Hermite_higher_dim} below. 

\begin{table*}
\begin{align*}
1_{(0,0)\in\mathcal F}(x,y) & =  \frac23 H_0-\frac13   H_2(y)\\
1_{(\sqrt{3},\sqrt{3})\in\mathcal F}(x,y) & = \frac1{12}  H_0+  \frac1{12} \sqrt{3} H_1(x) +  \frac1{12} \sqrt{3}  H_1(y)+   \frac1{12} H_1(x) H_1(y)+  \frac1{12}  H_2(y)\\
1_{(\sqrt{3},-\sqrt{3})\in\mathcal F}(x,y) & =  \frac1{12}  H_0-  \frac1{12} \sqrt{3} H_1(x) +  \frac1{12} \sqrt{3}  H_1(y)-  \frac1{12} H_1(x) H_1(y)+  \frac1{12}  H_2(y)\\
1_{(-\sqrt{3},\sqrt{3})\in\mathcal F}(x,y) & = \frac1{12}  H_0+  \frac1{12} \sqrt{3} H_1(x) - \frac1{12} \sqrt{3}  H_1(y)-  \frac1{12} H_1(x) H_1(y)+  \frac1{12}  H_2(y)\\
1_{(-\sqrt{3},-\sqrt{3})\in\mathcal F}(x,y) & =  \frac1{12}  H_0-  \frac1{12} \sqrt{3} H_1(x) -  \frac1{12} \sqrt{3}  H_1(y)+   \frac1{12} H_1(x) H_1(y)+  \frac1{12}  H_2(y)
\end{align*}
\caption{Indicator functions for Example \ref{Ex_Hermite-FractionDefinViaPolys}}
  \label{table:Ex_Hermite-FractionDefinViaPolys}
\hrulefill
\end{table*}

\begin{example} \label{confounding_Example_Hermite_higher_dim}

Consider a square grid of size $n$, $\design_{nn}$, and a polynomial $f$ of degrees in $x$ and in $y$ smaller than $n$, the Hermite polynomials and the standard normal distribution. Then we can write 
\begin{equation*} f(x,y) =\sum_{i,j=0}^{n-1} c_{ij} H_i(x) H_j(y) \end{equation*}
As both the degree in $x$ of $fH_k$  and the degree in $y$ of $fH_h$ are smaller than $2n-1$, we have 
\begin{gather*}
\expect\left( f(Z_1,Z_2) H_k(X_1) H_h(X_2) \right)  = c_{hk}   || H_k(X_1) ||^2   || H_h(X_2) ||^2 \\
c_{kh} = \frac1{ k! h!}
\sum_{(x,y)\in \design_{nn}} f(x,y) H_k(x) H_h(y) \lambda_x \lambda_y  
\end{gather*} 
Note if $f$ is the indicator function of a fraction $\mathcal F\subset \design_{nn}$ then
\begin{equation*} c_{kh}  = \frac1{k! h!} \sum_{(x,y)\in \mathcal F} H_k(x) H_h(y) \lambda_x \lambda_y  \qquad \text{with } 0\leq  h,k < n\end{equation*}
\end{example}

Example~\ref{Ex_Hermite-FractionDefinViaPolys} deals with a general design and introduces  the more general theory of Section~\ref{section:genericNoRing}.

\begin{example} 
\label{Ex_Hermite-FractionDefinViaPolys}
 Let  $\mathcal F$ be the zero set of 
\begin{equation*} \left\{
\begin{aligned}
g_1&=x^2-y^2=H_2(x)-H_2(y) = 0\\
g_2&=y^3-3y=H_3(y) = 0\\
g_3&=xy^2-3x=H_1(x) \left(H_2(y)-2H_0\right) =0\end{aligned}
\right.\end{equation*}
namely $\mathcal F$ is given by the five points $(0,0)$, $(\pm \sqrt{3},\pm \sqrt{3}) $.
Write a polynomial $f\in \reals[x,y]$ as $f=\sum q_i g_i + r$ where $r(x,y)=f(x,y)$ for $(x,y)\in \mathcal F$  
and
\begin{multline*}
  r \in \operatorname{Span}\left( H_{0}, H_1(x), H_1(y), H_1(x)H_1(y), H_2(y) \right) = \\ \operatorname{Span}\left( 1,x,y,xy, y^2 \right) .
\end{multline*}

If, furthermore, $f$ is such that
\begin{align*}
q_1(x,y) & =a_0+a_1 H_1(x)+a_2H_1(y)+a_3H_1(x)H_1(y) \\
q_2 & = \theta_1(x)+\theta_2(x)H_1(y)+\theta_3(x)H_2(y) \\
q_3 & = a_4+a_5H_1(y)
\end{align*}
with $a_i, \theta_j\in \reals$ for $i=0,\ldots,5$ and $j=1,\ldots,3$, 
then
\begin{equation*}
  \expect(g_i(Z_1,Z_2) q_i(Z_1,Z_2))=0, \quad i=1,2,3
\end{equation*}
for $Z_1$ and $Z_2$ independent normally distributed random variables.
Write $r$ as a linear combination of the indicator functions of the points in 
$\mathcal F$, i.e.
\begin{equation*}
  r(x,y)=\sum_{(a,b)\in \mathcal F}f(a,b) 1_{(a,b)\in\mathcal F}(x,y)
\end{equation*}

Each  indicator function $1_{(a,b)\in\mathcal F}$ belongs to
\begin{equation*}
  \operatorname{Span}\left( H_{0}, H_1(x), H_1(y), H_1(x)H_1(y), H_2(y) \right) 
\end{equation*}
and are shown in Table~\ref{table:Ex_Hermite-FractionDefinViaPolys}. Their expected values are given by the $H_0$-coefficients. Furthermore, by linearity
\begin{multline*}
  \expect(f(Z_1,Z_2))=\expect(r(Z_1,Z_2)) = \\ \sum_{(a,b)\in \mathcal F}f(a,b) \expect( 1_{(a,b)\in\mathcal F}(Z_1,Z_2))
\end{multline*}
and we can conclude 
\begin{multline*}
\expect(f(Z_1,Z_2))= \expect(r(Z_1,Z_2)) =
\displaystyle 2\frac{ f(0,0)}3 + \\ \displaystyle\frac{ f(\sqrt{3},\sqrt{3})+ f(\sqrt{3},-\sqrt{3})+ f(-\sqrt{3},\sqrt{3})+ f(-\sqrt{3},-\sqrt{3}) }{12}
\end{multline*}
\end{example} 

The key points in Example~\ref{Ex_Hermite-FractionDefinViaPolys} are
\begin{enumerate}
\item determine the class of polynomial functions for which $\expect(g_i(Z_1,Z_2) q_i(Z_1,Z_2))=0$ and
\item determine the $H_0$-coefficients of the indicator functions of the points in $\mathcal F$. 

\end{enumerate}
In Section~\ref{section:genericNoRing} we give algorithms to do this for any fraction $\mathcal F$.


\section{Higher dimension: general design support} \label{section:genericNoRing}

In the previous sections we considered particular designs whose sample points were zeros of orthogonal polynomials. 
In the Gaussian case we 
exploited the ring structure of the set of functions defined over the design in order to obtain recurrence formula and to write Fourier coefficients of higher order Hermite polynomials in terms of those of lower order Hermite polynomials (Example~\ref{confounding_Example_Hermite_one_dim}). Also we deduced a system of polynomial equations whose solution gives the weights of a quadrature formula. The mathematical tools that allowed this are Equation~(\ref{malliavinderivative}) and the particular structure it implies for Hermite polynomials on the recurrence relation for general, orthogonal polynomials
\begin{equation} \label{recurrence_general}
\pi_{k+1} (x) = (\gamma_k  x -\alpha_k) \pi_k(x) - \beta_k \pi_{k-1}(x)
\qquad x\in \reals
\end{equation}
with $\gamma_k, \alpha_k\neq 0$ and $\alpha_k \gamma_k \gamma_{k-1}>0$ (cf. Theorem \ref{Favard}).

In this section we switch focus and consider  a generic set of points in $\reals^d$ as a design, or nodes for a cubature formula, and a generic set of orthogonal polynomials.  We gain something and lose something. 
The essential computations are  linear: 
such is the computation of a Gr\"obner basis for a finite set of distinct points~\cite{MR680050}; 
the Buchberger M\"oller type of algorithm in Table~\ref{BuchbergerMollerForOrthogonalPolynomial} is based on finding solutions of linear systems of equations; 
in Section~\ref{section:c1} we give a characterisation of polynomials with the same expected values which is a linear expression of some Fourier coefficients and a square free polynomial of degree two in a larger set of Fourier coefficients (see Equation~\ref{cond}).

Given a set of points and a term-ordering 
the algorithm in Table~\ref{BuchbergerMollerForOrthogonalPolynomial} 
returns the reduced Gr\"obner basis of the design ideal expressed as linear combination of orthogonal polynomial of low enough degree.
It does so directly; that is, 
it computes the Gr\"obner basis by working only in the space of orthogonal polynomials.

We lose the equivalent of Theorem~\ref{productHermite} for Hermite polynomials, in particular we do not know yet how to  impose a ring structure on $\operatorname{Span}(\pi_0,\ldots, \pi_n)$ for generic orthogonal polynomials $\pi$ and we miss a general formula to write the product $\pi_k \pi_n$ as linear combination of $\pi_{i}$ with $i=0,\ldots,n\wedge k , n+k$, which is fundamental for the aliasing structure discussed for Hermite polynomials. 

For multivariate cubature formul{\ae} we refer e.g. to~\cite{MR1768956} and~\cite{MR0907119} which, together with~\cite{MR680050}, are basic references for this section. For clarity we repeat some basics and notation.
Let  $\lambda$ be a one-dimensional probability measure and $\{ \pi_n \}_{ n\in \mathbb Z_{\geq 0}}$ be its associated orthogonal polynomial system. 
To a multi-index $\alpha=(\alpha_1,\ldots,\alpha_d)\in \mathbb Z_{\geq 0}^d$ we associate the monomial $x^\alpha=x_1^{\alpha_1}\cdots  x_d^{\alpha_d} $  and the  product of polynomials  $\pi_\alpha(x)= \pi_{\alpha_1}(x_1) \dots \pi_{\alpha_d}(x_d)$. 
Note that $\{ \pi_\alpha \}_\alpha$ is a system of orthogonal polynomials for the product measure $\lambda^d$.
 Theorem~\ref{Correspondence_Monomials_HermitePolynomials} describes the one-to-one correspondence between the $x^\alpha$'s and the $\pi_\alpha(x)$'s.

\begin{theorem} \label{Correspondence_Monomials_HermitePolynomials}
\begin{enumerate} 
\item \label{1} 
For $d=1$ and $k\in \mathbb Z_{\geq 0}$, in the notation of Equation~(\ref{recurrence_general})
 we have that 
\begin{equation*}
x^{k} = \sum_{j=0}^{k} c_{j}(x^k) \pi_{j}(x) 
\end{equation*}
where  $c_0(x^0)=1$, $ c_{-1}(x^0)=c_{1}(x^0) =0$, and, for $k=1,2,\dots$  and $ j=0, \dots, k-1$
 \begin{eqnarray*}
c_{-1}(x^k)&=&c_{k+1}(x^k) =0\\
c_j(x^k) &=& \frac{c_{j-1}(x^{k-1})}{\gamma_{j-1}} + \frac{c_j(x^{k-1})\alpha_j}{\gamma_j} + \frac{c_{j+1}(x^{k-1})\beta_{j+1}}{\gamma_{j+1}}\\
c_k(x^k) &=& \frac 1{\gamma_0\dots \gamma_{k-1}} 
\end{eqnarray*}
\item \label{2}
For $d>1$, the monomial $x^\alpha$ is a linear combination of $\pi_\beta$,  with $\beta \le \alpha$ component wise, and vice versa.
In formul{\ae}
\begin{eqnarray}
\pi_\alpha= \sum_{\beta \le \alpha} a_\beta x^\beta  \qquad and \qquad  
x^\alpha =  \sum_{\beta \le \alpha} b_\beta \pi_\beta 
\end{eqnarray}
where $\beta \le \alpha $ holds component wise.
\end{enumerate}
\end{theorem}

\begin{proof} 
The proof of Item~\ref{1} is by induction and that of Item~\ref{2} follows by rearranging the coefficients in the product.
They are given in Appendix~\ref{Proofs}. \qed   \end{proof}

\begin{example}\label{cor_her} If  $\pi_j$  is the $j$-th Hermite polynomial, then 
Item~1 of Theorem~\ref{Correspondence_Monomials_HermitePolynomials} gives the well known result
\begin{align*}
c_j(x^k) &= 0   &  \text{if   $k+j$ is  odd} \\
c_j(x^k)&= \left ( \begin{array}{c} k \\j \end{array}\right)(k-j-1)!!  & \text{if  $k+j$  is even}
\end{align*} 
\end{example}

Direct application of Theorem~\ref{Correspondence_Monomials_HermitePolynomials} is cumbersome and we need only to characterise the polynomial functions for which the cubature formula is exact. So we proceed by another way. 
The finite set of distinct points $\design\subset \reals^d$ is associated to its vanishing polynomial ideal 
\begin{equation*}
\mathcal I(\design) 
=\left\{ f\in \reals[x]: f(z) =0 \text{ for all } z\in \design  \right\}
\end{equation*} 
Let $LT_\sigma(f)$ or $LT(f)$ denote the largest term in a polynomial $f$ with respect to a term-ordering $\sigma$. 
Let $[f(z)]_{z\in \design}$ be the evaluation vector of the polynomial $f$ at $\design$
and for a finite set of polynomials $G\subset \reals[x]$ let $[g(z)]_{z\in \design, g\in G}$ be the evaluation matrix whose columns are the evaluation vectors at $\design$ of the polynomials in $G$. 
In DoE often this matrix is called the $X$-matrix of $\mathcal D$ and $G$.

As mentioned at the end of Section~\ref{section:generalframework}, the space $\mathcal L(\design)$ of real valued functions defined over $\design$ is a linear space and particularly important vector space bases can be constructed as follows. 
Let $  LT(\mathcal I(\design))  = \langle LT_\sigma(f) : f\in \mathcal I(\design) \rangle$.
If $G$ is the $\sigma$-reduced Gr\"obner basis of $\mathcal I(\design)$, then $ LT(\mathcal I(\design) ) = \langle LT_\sigma(f) : f\in G \rangle$.
Now we can define two interesting vector space bases of $\mathcal L(\design)$. 
Let   $L=\{ \alpha \in \mathbb Z_{\geq 0}^d  : x^\alpha \not\in  LT(\mathcal I(\design))  \}$ and define
\begin{equation*}
\mathcal B=\{ x^\alpha : \alpha  \in  L  \}
\qquad \text{ and } \qquad \mathcal {OB} =\{ \pi_\alpha : \alpha  \in  L  \}
\end{equation*}
\begin{example} \rm 
For $d=2$ and $L=\{ (0,0),(1,0),(0,1),(2,0) \}$ we have $\mathcal B=\{ 1, x,y,x^2\}$ and 
 $\mathcal{OB}= \{1, \pi_1(x),\pi_1(y),\pi_2(x)\}$, since $\pi_0(x)=\pi_0(y)=1$.
 \end{example}
 
The sets $L$, $\mathcal B$ and $\mathcal {OB}$ depend on $\sigma$. 
It is well known that  if $t \in \mathcal{B}$ and $r$ divides $t$, then $r\in \mathcal{B}$; it follows that if $\alpha\in L$ and $\beta \le \alpha$ component wise then also $\beta$ belongs to $L$ and $\pi_\beta$ to $\mathcal {OB}$.
Note that $\sigma$ induces a total ordering also on the orthogonal polynomials: $\pi_\alpha
< _\sigma \pi_\beta$ if and only if $x^\alpha<_\sigma x^\beta$; analogously $\alpha <_\sigma \beta$ if and only if $x^\alpha <_\sigma x^\beta$ for each   $\alpha, \beta \in \mathbb Z_{\ge 0}^d$. 
Here we used the same symbol to indicate related orderings over the $\alpha$'s, the $x^\alpha$'s and the $\pi_\alpha$'s. Further, given  
$\alpha \le \beta$ componentwise, since $x^\alpha$ divides $x^\beta$ and 
since $1\leq_\sigma x^{\beta-\alpha}$,  we have $x^\alpha \le _\sigma x^\beta$, that is $\alpha \le_\sigma \beta	 $.

Now, given a term-ordering $\sigma$, 
any $g\in G$ can be uniquely written as its leading term, $x^\alpha=LT(g)$, and tail which is a linear combination of terms in $\mathcal B$ preceding   $LT(g)$ in $\sigma$, that is $g=x^\alpha+\sum_{\beta \in L, \beta <_\alpha \alpha} a_\beta x^\beta $ with   $a_\beta\in \reals$. 

Theorem~\ref{RemainderOverHermiteBasis}  provides an alternative to the classical method of rewriting a polynomial $f$ in terms of orthogonal polynomials. It does so by applying Theorem~\ref{Correspondence_Monomials_HermitePolynomials} and by substituting each monomial in $f$.
Theorem~\ref{RemainderOverHermiteBasis} gives linear rules to write the elements of $G$ and the remainder of a polynomial divided by $G$ as linear combinations of orthogonal polynomials of low enough order. The proof is in Appendix~\ref{Proofs}.

\begin{theorem}  \label{RemainderOverHermiteBasis}

\noindent
\begin{enumerate}
\item  $\operatorname{Span}(\mathcal B) = \operatorname{Span}(\mathcal{OB})$;
\item Let $G$ be the reduced $\sigma$-Gr\"obner basis of $\mathcal I(\design)$. Each  $g\in G$ with $LT(g)=x^\alpha$ is uniquely written as 
\begin{equation*}
g= \pi_\alpha - \sum_{\beta \in L , \beta <_\sigma \alpha} b_\beta \pi_\beta
\end{equation*}
where  $b=[b_\beta]_{\beta \in L, \beta <_\sigma \alpha}$ solves the linear system
\begin{equation*}
  \left[\pi_\beta(z) \right]_{z\in \design, \beta \in L, \beta <_\sigma \alpha} b = \left[ \pi_\alpha(z) \right]_{ z\in \design};
\end{equation*}
in words the coefficient matrix is the evaluation matrix over $\design$ of the orthogonal polynomials $\pi_\beta$ with $x^\beta$ in tail of $g$
and the vector of constant terms is the evaluation vector of $\pi_\alpha$.
\item Let $p\in \reals[x]$  be a polynomial and $ [p(z)]_{z\in \design}$ its evaluation vector.
The polynomial $p^*$ defined as 
\begin{eqnarray}\label{int_herm}
 p^*= \sum_{\beta \in L} a_\beta \pi_\beta
 \end{eqnarray}
where  $a=[a_\beta]_{\beta\in L}$ solves the linear system
\begin{equation*}
  \left[ \pi_\beta(z) \right]_{z\in \design,  \beta \in L} a = \left[ p(z) \right]_{ z\in \design}
\end{equation*}
is the unique polynomial belonging to $ \operatorname{Span}(\mathcal{OB})$ such that
 $p^*(z)=p(z)$ for all $z\in \design$.
\end{enumerate}
 \end{theorem}

Theorem~\ref{RemainderOverHermiteBasis} 
provides a pseudo-algorithm to compute a Gr\"obner basis for $\mathcal I(\design)$  and the interpolating polynomial at  $\design$ in terms of orthogonal polynomials of low order directly from $\design$ and $\mathcal{OB}$. 
Table~\ref{BuchbergerMollerForOrthogonalPolynomial} gives the algorithm which is a variation of the Buchberger-M\"oller algorithm~\cite{MR680050}. It starts with a finite set  of distinct points $\design$ and a term-ordering $\sigma$ and it returns $L$ and the expressions $g= \pi_\alpha - \sum_{\beta \in L, \beta <_\sigma \alpha} b_\beta \pi_\beta$ for $g$ in the reduced $\sigma$-Gr\"obner basis of $\mathcal I(\design)$. It does so by performing linear operations.
If the real vector $[p(z)]_{z\in \design}$ is assigned, then the expression  $p^*= \sum_{\beta\in L} a_\beta \pi_\beta$ can now be found using Item~3 in Theorem~\ref{RemainderOverHermiteBasis}. This permits to rewrite every polynomial $p\in \reals[x]$ as a linear combination of orthogonal polynomials. 

The algorithm in Table~\ref{BuchbergerMollerForOrthogonalPolynomial} returns the $\sigma$-reduced Gr\"obner basis as linear combination of orthogonal polynomials. 
It performs  operations only with orthogonal polynomials and in particular it does not involve at any step the $x^\alpha$ monomials.
This is computationally faster than first computing a classical Gr\"obner basis in the $x^\alpha$ and next substituting the $\pi_\alpha$.
Furthermore working with only one vector space basis, and not switching between the $x^\alpha$ and the $\pi_\alpha$, is conceptually appealing.

\begin{table}[h] 
\hrulefill
\begin{description}
\item[\bf Input:] a set $\design$ of distinct points in $\reals^d$, a term-ordering $\sigma$ and any vector norm $||\cdot ||$.
\item[\bf Output:] the reduced $\sigma$-Gr\"obner basis $G$ of  $\mathcal I(\design)$ as linear combination of orthogonal polynomials and the set $L$.
\item[Step~1] Let  $L=\{0 \in Z_{\geq 0}^d  \}$, $\mathcal{OB}=[ 1]$, $G=[\;]$ and $M=[x_1,\dots,x_d]$.
\item[Step~2] If $M=[\;]$ stop; else set $x^\alpha=\min_{\sigma}(M)$ and delete $x^\alpha$ from $M$.
\item[Step~3] Solve in $b$ the overdetermined linear system 
$\left[\pi_\beta(z)\right]_{z\in \design,  \beta \in L} b = \left[ \pi_\alpha(z) \right]_{z\in \design}$ and compute the residual 
$$\rho=  \left[ \pi_\alpha(z) \right]_{z\in \design} - \left[\pi_\beta(z)\right]_{z\in \design,  \beta \in L} b $$
\item[Step~4] $\phantom{nn}$ \\
	\begin{enumerate}
		\item If $\|\rho\| > 0$, then include $\alpha$ in $L$, 
		and include in $M$ those elements of 
		 $\{x_1 x^\alpha, \ldots,x_d x^\alpha\}$ which are not multiples of an element in~$M$ or of $LT(g)$, $g\in G$. Return to Step~2. 
	\item If $\|\rho\| = 0$, then include  in $G$ the polynomial 
\begin{equation*}
g=\pi_\alpha - \sum_{\beta \in L} b_\beta \pi_\beta
\end{equation*}	
	where the	 values $b_\beta$, $\beta \in L$, are the components of the solutions $b$ of the linear system in Step~3. Delete from~$M$  all multiples of~$x^\alpha$.
	\end{enumerate}
\end{description}
\caption{Buchberger-M\"oller algorithm using orthogonal polynomials}
\label{BuchbergerMollerForOrthogonalPolynomial}
\hrulefill
\end{table}

Summarising: given a function $f$, a finite set of distinct points $\design\subset \reals^d$ and a term-ordering $\sigma$, a probability product measure $\lambda^d$ over $\reals^d$, its  system of product orthogonal polynomials, and a random vector with probability distribution $\lambda^d$, then
the expected value of $f$ with respect to $\lambda^d$ can be approximated by
\begin{enumerate}
\item computing $L$ with the algorithm in Table~\ref{BuchbergerMollerForOrthogonalPolynomial} and
\item determining, by solving the linear system 
$$\left[ \pi_\beta(z) \right]_{z \in \design  , \beta\in L} a =\left[ f(z) \right]_{z\in \design},$$
  the unique polynomial $p^*$ such that $p^*(z)=f(z)$ for all $z\in\design$. The polynomial $p^*$ is expressed as linear combination of orthogonal polynomials.
\item The coefficient $a_0$ of $\pi_0$ is the required approximation.
\end{enumerate} 
Recall that $p^*(x)=\sum_{z\in \design} f(z) l_z(x)$ is a linear combination of the indicator functions of the points in $\design$ (Lagrange polynomials) and hence
$a_0 = \sum_{z\in \design} f(z) \expect \left( l_z(X) \right)$. In particular, $\expect \left( l_z(X) \right)=\lambda_z$, $z\in \design$, can be computed by applying the above to $f=l_z$. Notice however that as $\lambda^d$ is a product measure, the $\lambda_z$ can be obtained from the one-dimensional ones as noticed before Theorem~\ref{higherdim:theorem}.
It would be interesting to generalise this section to non-product measures.

Here an algorithm has been provided to approximate the expected value of polynomials. 
Next the set of  polynomials whose expected value coincides with the value of the cubature formula is characterised.
In Section~\ref{section:c1} we provide a characterisation of the full set 
 via linear relationships that Fourier coefficients of suitable polynomials have to satisfy, 
 while in Section~\ref{section:c2} a possibly proper subset has been characterised via a simple condition on the total degree of the polynomials.

\subsection{Characterisation of polynomial functions with zero expectation}\label{section:c1} 

\begin{table*}\begin{equation*} \begin{array}{llllll}
 c_0^{(1)}= -34 \quad &  c_{(1,0)}^{(1)} =0  \quad &    c_{(0,1)}^{(1)}=-2  \quad &  c_{(0,0)}^{(1)}=-8  \\
 c_0^{(2)}=0    \quad &  c_{(2,0)}^{(2)} =0  \quad & c_{(1,1)}^{(2)}=1      \quad &  c_{(1,0)}^{(2)} =-2   \quad &    c_{(0,1)}^{(2)}=1  \quad &  c_{(0,0)}^{(2)}  = 2\\
 c_0^{(3)}=10    \quad &  c_{(2,0)}^{(3)} =2 \quad &  c_{(1,1)}^{(3)}=0     \quad &   c_{(1,0)}^{(3)} =-5  \quad &    c_{(0,1)}^{(3)}=1 \quad & c_{(0,0)}^{(3)}  = 5863/2987
\end{array}  \end{equation*}
\caption{A solution for \eqref{cond}\label{table:cond}}
\hrulefill
\end{table*}

In this section we characterise the set of polynomials whose expected value coincides with the value of the cubature formula.
As mentioned in Section~\ref{section:generalframework} given $\design\in \mathbb R^d$, its vanishing ideal $\mathcal I(\design)$, a term-ordering $\sigma$ and the Gr\"obner basis $G$ of $\mathcal I(\design)$ with respect to $\sigma$, then any polynomial $f\in\reals[x]$ can be written as 
\begin{equation*}
f(x)= \sum_{g\in G} q_g(x) g(x) + r_\sigma(x)   
\end{equation*}
where $r_\sigma(x)$ is  unique in $\operatorname{Span}(\mathcal B)$ such that $r_\sigma(z)=f(z)$ for all $z\in \design$ and can be written as $r_\sigma(x) = \sum_{z\in \design}f(z) l_z(x)$, where the $l_z$'s are the product Lagrange polynomials in Section~\ref{section:HigherDimension}. 
Theorem~\ref{RemainderOverHermiteBasis} states how to write $r_\sigma$ over $\mathcal{OB}$. 

If $f\in \reals[x]$ is such that  $\expect_{\lambda} \left( f(X) \right) = \expect_{\lambda} \left( r_\sigma(X) \right)$  then
we have $ \expect_{\lambda} \left( f(X) - r_\sigma(X) \right) = 0
$, 
where $f-r_\sigma\in \mathcal I(\design)$. 
Hence to study the set
\begin{equation*}
\mathcal E_\sigma =\left\{ f\in \reals[x]: \expect_{\lambda} \left( f(X) \right) = \expect_{\lambda} \left( r_\sigma(X) \right) \right\}
\end{equation*}
is equivalent to characterize the set
\begin{equation*}
\mathcal E_0= \left\{ p\in \mathcal I(\design): \expect_{\lambda} \left( p(X) \right) =0 \right\}
\end{equation*}
We do this in two ways. 
First we study the Fourier expansion of the elements of $\mathcal E_0$, 
next we present some results about the degree of the elements belonging to $\mathcal E_\sigma$. 

The elements of $\mathcal E_0$ are characterized in Theorem~\ref{mio}. 
Note that if $f\in\reals[x]$ is such that
 $f=p+r_\sigma$ with $p\in \mathcal E_0$ 
 and $r_\sigma \in \operatorname{Span}(\mathcal {B})$ then  by linearity and independence 
 \begin{equation*}
   \expect_{\lambda}(f)=\sum_{(z_1,\ldots,z_n) \in \design} f(z_1,\ldots,z_n) \lambda_{
z_1}^{n_1}\cdots \lambda_{z_d}^{n_d}
 \end{equation*}
\begin{theorem}\label{mio}

Let $\lambda$ be a product probability measure with product orthogonal polynomials $\pi_\alpha(x)$, $\alpha\in \mathbb Z^d_{\geq 0}$ and let $X$ be a random vector with distribution $\lambda^d$. 
Let $\design\subset \reals^d$ be a set of distinct points, $\sigma$ a term-ordering, 
$G$ the $\sigma$-reduced Gr\"obner basis of $\mathcal I(\design)$ whose elements as linear combinations of orthogonal polynomials. Thus for $g\in G$ and $x^{\alpha}=LT(g)$ we write
 $$g=\pi_{\alpha} - \sum_{\alpha >_\sigma \beta \in L} c_\beta(g) \pi_\beta
 $$
 where  ${\alpha >_\sigma \beta \in L}$ stands for $\alpha >_\sigma\beta$ and $\beta \in L$.
 
Let $p =  \sum_{g\in G} q_g g \in \mathcal I(\design)$ for suitable $q_g\in \reals[x]$, and consider the Fourier expansion of each $q_g$ , $g\in G$, 
\begin{eqnarray} \label{strutturaq}
q_g&=& \sum_{ \beta\in\mathbb Z^d_{\geq 0} }        c_\beta(q_g)  \pi_\beta 
\end{eqnarray}

Then $ \expect_{\lambda}\left( p(X) \right)= 0 $ if and only if 
\begin{eqnarray}\label{cond}
 \sum_{g\in G} \|\pi_{\alpha}\|_\lambda^2 c_{\alpha}(q_g) -   \sum_{g\in G} \sum_{\alpha >_\sigma \beta \in L}  \|\pi_\beta\|_\lambda^2 c_\beta(q_g) c_\beta(g) =0
\end{eqnarray}
\end{theorem} 
\begin{proof}
The key observation is that $\expect_{\lambda}\left( \pi_m \ \pi_n \right) =0 $ if $n\neq m$ and then linearity of $\expect_{\lambda}$ is used.  The proof can be found in Appendix~\ref{Proofs}.
\qed   \end{proof}
 Importantly, only terms of low enough Fourier order  in Equation~(\ref{strutturaq}) matter for the computation of the expectation. 

\begin{example}\label{ex_zero_exp}
Consider $Z_1$ and $Z_2$ two independent standard normal random variables and hence the Hermite polynomials.
Consider also the five point design 
$$\design=\left\{ (-6,-1), (-5,0), (-2,1), (3,2),(10,3) \right\}$$
 and the $\sigma$$=${\tt{DegLex}} term-ordering {over the monomials in $\reals[x,y]$}. The algorithm in Table~\ref{BuchbergerMollerForOrthogonalPolynomial} gives $G=\{g_1,g_2,g_3\} $ where
\begin{eqnarray*}
g_1 & = &H_2(y) - H_1(x) +2 H_1(y) -4  \\
g_2 & =& H_2(x)H_1(y) -9    H_2(x) +47    H_1(x)H_1(y) -123    H_1(x)\\
& &+271    H_1(y) -399  \\
g_3 & =& H_3(x) -47    H_2(x) +300    H_1(x)H_1(y) -845    H_1(x)\\
&& +2040    H_1(y) -2987 
\end{eqnarray*} 
and $L=\{ (0,0), (1,0), (0,1) ,(1,1), (0,2)\}$, that is 
\[
\mathcal{OB} =\{1, H_1(y), H_1(x),  H_1(x)H_1(y), H_2(x)  \} 
\]
By Theorem~\ref{mio} for the purpose of computing its expectation a polynomial $p = q_1g_1+q_2g_2+q_3g_3 \in \mathcal I(\design)$ can be simplified to have the form 
\begin{multline*}
p  =  \left(  c_0^{(1)} H_2(y) +c_{(1,0)}^{(1)} H_1(x) +c_{(0,1)}^{(1)} H_1(y) +c_{(0,0)}^{(1)} \right) g_1  + \\ \left(  c_0^{(2)}H_2(x)H_1(y) + c_{(2,0)}^{(2)}H_2(x)+ c_{(1,1)}^{(2)} H_1(x)H_1(y) + \right. \\ \left. c_{(1,0)}^{(2)} H_1(x)
  + c_{(0,1)}^{(2)} H_1(y) + c_{(0,0)}^{(2)} \right)  g_2
+ \left( c_0^{(3)} H_3(x) + \right. \\ c_{(2,0)}^{(3)}H_2(x)+ c_{(1,1)}^{(3)} H_1(x)H_1(y) +c_{(1,0)}^{(3)} H_1(x) +\\ \left. c_{(0,1)}^{(3)} H_1(y) + c_{(0,0)}^{(3)} \right)     g_3 
\end{multline*}
and furthermore by Equation~(\ref{cond})
\begin{multline*}
c_0^{(1)}    2! - c_{(1,0)}^{(1)} + 2c_{(0,1)}^{(1)} -4c_{(0,0)}^{(1)} +c_0^{(2)} 2! -9    c_{(2,0)}^{(2)} 2! + 47 c_{(1,1)}^{(2)} \\
  -123 c_{(1,0)}^{(2)}  +271 c_{(0,1)}^{(2)} -399c_{(0,0)}^{(2)}  + c_0^{(3)}     3! -47   c_{(2,0)}^{(3)}  2!  \\
 + 300    c_{(1,1)}^{(3)} -845 c_{(1,0)}^{(3)}  +2040 c_{(0,1)}^{(3)} -2987    c_{(0,0)}^{(3)}  =0
\end{multline*}
In practice, for $i=1,2,3$, put  coefficients of $g_i$ and $q_i$ in two vectors, multiply them component wise and sum the result.
There are infinite polynomials that satisfy the above equations,
one such polynomial is given in Table~\ref{table:cond}. 
The above equation involves only a finite number of Fourier coefficients, namely
 $c_\beta(q_g)$'s is relevant for the equation if and only if
 $\beta\in L$ and $x^\beta$ is smaller in $\sigma$ than the leading terms of $g\in G$. 
Hence if to $q_g$ we add a polynomial of the form $\sum_{\beta>\alpha \text{ or } \beta\not\in L} c_\beta H_\beta$
we still obtain a zero mean  polynomial. That is, we can modify $q_g$ by adding high enough terms without changing the mean value.

For example by adding $H_4(x)$ to $q_1$ and $ H_4(y)$ to $q_2$ we obtain the following zero mean polynomial 
\begin{multline*}
 p(Z_1,Z_2) =\\ Z_1^2Z_2^5 + 10Z_1^6 + Z_1^4Z_2^2 - 9Z_1^2Z_2^4 + 47Z_1Z_2^5 - 469Z_1^5 + \\ 3002Z_1^4Z_2 + Z_1^3Z_2^2 - 6Z_1^2Z_2^3 - 123Z_1Z_2^4 + 270Z_2^5 - \\ 8614Z_1^4 + 20990Z_1^3Z_2 + 96Z_1^2Z_2^2 - 282Z_1Z_2^3 - 424Z_2^4 \\
 - 87898560/2987Z_1^3 - 6700Z_1^2Z_2 + 1389Z_1Z_2^2 - 1690Z_2^3 + \\ \frac{71785814}{2987}Z_1^2 - \frac{218275468}{2987} Z_1Z_2 + 4845Z_2^2 + \\ \frac{307862660}{2987}Z_1 - \frac{5937584}{2987}Z_2 - \frac{5931425}{2987}
\end{multline*}
\end{example}

\subsection{On exactness of cubature formul{\ae}}\label{section:c2}
Here we adopt another viewpoint and  characterise the set $\mathcal E_\sigma$. 
Instead of studying the Fourier expansion of its polynomials, we focus our attention on their degree. 
Given a degree compatible term ordering $\sigma$, we show how to compute  the maximum degree $s$ such that
$\mathbb R[x]_{\leq s}$ is in $\mathcal E_\sigma$, that is the degree of the cubature formula with nodes $\design$.
Our strategy is based on the definition of $s$-orthogonal polynomials~\cite{MR0907119}.

\begin{definition}
A polynomial $g\in \reals[x] $ is $s$-orthogonal if $s \in \mathbb Z_{>0}$ is the maximum integer such that
$$ fg \in \reals[x]_{\leq s} \;\;\;\; \text{ implies} \;\;\;\;\expect_{\lambda} \left( f(X)g(X)\right) =0$$
Furthermore, a set $G$ of polynomials is $s$-orthogonal if each $g\in G$  is $s(g)$-orthogonal and $s= \min_{g\in G} s(g)$.
\end{definition}

Theorem~\ref{Teo_Grado} reformulates and summarizes two theorems about the degree of a cubature formula presented in~\cite{MR1768956} and~\cite{MR0907119} (where H-bases are considered). 

\begin{theorem}\label{Teo_Grado}
Given a set $\design$ and a degree compatible term ordering $\sigma$  the following conditions are equivalent.
\begin{enumerate}
\item $\reals[x]_{\leq s} \subset \mathcal E_\sigma $;
\item $\expect_{\lambda} \left( f(X) \right)=0$ for all $f \in \mathcal I(\design) \cap \reals[x]_{\leq s}$;
\item the $\sigma$-Gr\"ob\-ner basis $G$ of $\mathcal I(\design)$ is $s$-orthogonal.
\end{enumerate} 
\end{theorem}
\begin{proof}
$1 \Rightarrow 2$. Let  $f \in \mathcal I(\design) \cap \reals[x]_{\leq s} $. Since by hypothesis $\reals[x]_{\leq s} \subset \mathcal E_\sigma $, then $f \in \mathcal E_\sigma$ that is $\expect_{\lambda} \left( f(X) \right) = \expect_{\lambda} \left( r_\sigma(X) \right)$. Since $f \in I (\design)$, we have $r_\sigma(x) \equiv 0$ and so $ \expect_{\lambda}\left( f(X) \right)=0$. \\
$2 \Rightarrow 3$. For each $g \in G$ let $f$ be such that $fg \in \reals[x]_{\leq s}$. Since $fg \in  \mathcal I(\design) \cap \reals[x]_{\leq s}$ then 
$\expect_{\lambda}\left( f(X)g(X) \right)=0$ and so $g$ is $s$-orthogonal. \\
$3 \Rightarrow 1$. 
 For $p\in\reals[x]_{\leq s}$, we have that
$$p=\sum_{g \in G} gq_g + r_\sigma$$
where each $ gq_g$  is such that $LT_\sigma(g q_g) \le LT_\sigma(p)$ and so, since $\sigma $ is degree compatible, $\deg(g q_g) \le \deg(p)\le s$. It follows that, since $G$ is $s$-orthogonal, $\expect_{\lambda}\left(gq_g \right) =0 $.  By linearity we obtain $\expect_{\lambda}\left( p(X) \right) =
\expect_{\lambda}\left( r_\sigma(X) \right)  $, that is $p \in \mathcal E_\sigma$.
\qed   \end{proof}

\begin{remark}
The maximum integer  $s$ such that  $\expect_{\lambda}\left( f(X) \right) =0 $ for each $f \in \mathcal I(\design) \cap \reals[x]_{\leq s}$ is the degree of the cubature formula with nodes $\design$ and with respect to $\sigma$.
\end{remark}

Theorem~\ref{Teo_Grado} shows that the maximum integer $s$ such that all polynomials of total degree $s$ are in $\mathcal E_\sigma$
coincides with the maximum 
$s$ such that $G$ is $s$-orthogonal. Hence we focus our attention on the $s$-orthogona\-li\-ty of the elements of $G$. 

\begin{theorem}\label{bounds}
Each polynomial $g \in\reals[x]$ is $s$-orthogonal with $\deg(g)-1 \le s< 2\deg(g)$.
\end{theorem}
\begin{proof}
Since $fg \in \reals[x]_{\le (\deg(g)-1)}$ if and only if $f$ is the identically zero polynomial, then $fg \in \reals[x]_{\le (\deg(g)-1)}$ implies $\expect_{\lambda}\left( f(X)g(X)\right) =0$, that is  $g$ is always $s$-orthogonal with $s\ge \deg(g)-1$. Moreover, $s < 2\deg(g)$; in fact, $g^2$ belongs to $ \reals[x]_{\le 2\deg(g)}$ and, from the orthogonality of the polynomials $\pi_\alpha$, we have
$$\expect_{\lambda}\left( g(X)g(X)\right) = \sum_{\alpha} (c_\alpha(g))^2 \|\pi_\alpha\|^2$$
As $g$ is not identically zero, then $\expect_{\lambda}\left( g(X)g(X)\right) \neq 0$ and so $g$ cannot be $2\deg(g)$-orthogonal. \qed   \end{proof}

The following theorem shows how to detect the $s$-ortho\-go\-na\-lity of a polynomial analysing its Fourier coefficients.

\begin{theorem}\label{pol_s_ort}
Let $ g = \sum_{\alpha} c_\alpha(g) \pi_\alpha$  the Fourier expansion of a polynomial~$g\in\mathbb R[x]$.
\begin{enumerate}
\item If  $c_0(g) \neq 0$, the polynomial $g$ is $(\deg(g)-1)$-orthogonal. 
\item If  $c_0(g) =0$, the polynomial  $g$ is $s$-orthogonal where
$s$  is   such that $ c_\alpha(g) = 0$ for all $\alpha$ s.t. $ \sum_{i=1}^d \alpha_i \le s-\deg(g)$.
\end{enumerate}
\end{theorem}
\begin{proof}
\begin{enumerate}
\item  Let $f$ be a constant polynomial, $f\equiv 1$. We have that $fg \in  \reals[x]_{\le \deg(g)}$ implies  
$$\expect_{\lambda}\left( f(X)g(X)\right) = \expect_{\lambda}\left( g(X)\right) = c_0(g) \neq 0$$
and so $g$ is not $\deg(g)$-orthogonal. From Theorem~\ref{bounds} we conclude that $g$ is $(\deg(g)-1)$-orthogonal.
\item Let  $s < 2 \deg(g)$ and let $f$ be a polynomial such that $fg \in  \reals[x]_{\le s}$. Since  $\deg(f) \le s-\deg(g)$ the Fourier expansion of $f$ is such that  $f=\sum_{|\alpha| \le s-\deg(g)} c_\alpha(f) \pi_\alpha$.
From the orthogonality of the polynomials $\pi_\alpha$  and from  $ s-\deg(g) < \deg(g)$ we have 
$$\expect_{\lambda}\left(f(X) g(X)\right)  = \sum_{|\alpha| \le s-\deg(g)} c_\alpha(f) c_\alpha(g) \|\pi_\alpha\|^2$$
The generality of $f$ implies that  $\expect_{\lambda}\left( f(X)g(X)\right) =0$ only if 
$c_\alpha(g) = 0$ for each $\alpha$ such that $ |\alpha|\le s-\deg(g)$.
\end{enumerate}
\qed   \end{proof}

\begin{corollary}\label{finale_grado}
Given a finite set of distinct points $\design\in\mathbb R^d$ and a degree compatible term ordering $\sigma$ on $\mathbb R[x]$, let $G$ be the reduced $\sigma$-Gr\"obner basis of $\mathcal I(\design)$. Then the maximum integer $s_c$ 
such that $\mathbb R[x]_{\leq s_c}$ is in  $\mathcal E_\sigma$   is 
$$ s_c= \min_{g\in G} s(g)$$ where $s(g)$ is such that the Fourier expansion of each $g \in G$ is given by
$$g=\sum_{\alpha} c_\alpha(g) \pi_\alpha $$
with $c_\alpha(g) = 0 $ forall $\alpha$ such that $\sum_{i=1}^d \alpha_i  \le  s(g)-\deg(g)$.
\end{corollary}
\begin{proof}
By Theorem~\ref{Teo_Grado} the thesis follows if $G$ is $s_c$-orthogonal. But this is true since, from Theorem~\ref{pol_s_ort}, we have that  each $g \in G$ is $s(g)$-orthogonal.
\qed   \end{proof}

%

\begin{example}
Given the set of points
$$ \design=\{ (-1,0),\;(-1,-2),\; (1,-1+\sqrt{3}),\; (1,-1-\sqrt{3}),\;(2,1)\}$$
and the degree compatible term ordering $\sigma$$=${\tt{DegLex}}, then 
by e.g. the algorithm in Table~\ref{BuchbergerMollerForOrthogonalPolynomial} 
the reduced $\sigma$-Gr\"obner basis $G$ of $\mathcal I(\design)$ can be written as 
\[
G=\left \{ \begin{array}{rcl}
g_1 & = &   H_2(y) -H_1(x) +2H_1(y) \\
g_2 & = &   H_2(x) H_1(y) - H_2(x)\\
g_3 & = &   H_3(x) -2H_2(x) +2H_1(x) 
\end{array}\right.
\]  
Since for each $g \in G$,  $s(g)$ is such that $c_\alpha(g)=0$ for each $\alpha$ s.t. $\sum_{i=1}^d \alpha_i \le s(g)-\deg(g)$, we have that $s(g_1)=2$, $s(g_2)=4$ and $s(g_3)=3$.  It follows that  $G$ is $s_c$-orthogonal, with $s_c=\min\{s(g_1),s(g_2),s(g_3)\} =2$ and so, from Corollary~\ref{finale_grado} it follows that the maximum integer $s$ such that  $\mathbb R[x]_{\leq s}\subset \mathcal E_\sigma$ is~$2$.
\end{example}

\begin{example}
For $\design$ of Example~\ref{ex_zero_exp} and $\sigma$$=${\tt{DegLex}}, the reduced $\sigma$-Gr\"obner basis $G=\{g_1,g_2,g_3\}$ of $\mathcal I(\design)$ is such that  $s(g_1)=1$, $s(g_2)=2$ and $s(g_3)=3$ since $c_0(g) \neq 0$ for all $g\in G$.  It follows that  $G$ is $1$-orthogonal and so the cubature formula w.r.t. $\sigma$ and $\design$ is exact for all the polynomials with degree $0$ and $1$. Nevertheless let us remark that the cubature formula is exact for a much larger class of polynomials as shown in Theorem~\ref{mio} and in  Example~\ref{ex_zero_exp}.
\end{example}  

\section{Conclusion}

In this paper we mixed tools from Computational Commutative Algebra,  orthogonal polynomial theory and Probability to address the recurrent statistical problem of estimation of mean values of polynomial functions. 
Our work shares great similarity with applications of computational algebra to design and analysis of experiments, which inspired us with a non-classical viewpoint to cubature formulae.

We obtained two main results.
In the Gaussian case we obtained a system of polynomial equations whose solution gives the weights of a quadrature formula (Theorem~\ref{finalHermite}).
For a  finite product measure which admits an orthogonal system of polynomials, we characterise the set of polynomials with the same mean value. This depends substantially from Equation~(\ref{cond}) and it is in terms of Fourier coefficients of particular polynomials obtained by adapting 
Gr\"obner basis theory.

\section{Appendix: proofs} \label{Proofs}

{\bf Theorem~\ref{Correspondence_Monomials_HermitePolynomials}:}
\begin{proof} 1. 
The proof is by induction on the monomial degree~$k$.
From the three terms recurrence formula $\pi_{j+1} = (\gamma_j x -\alpha_j) \pi_j - \beta_j \pi_{j-1}$ we have
$$x \pi_{j} =\frac{\pi_{j+1}}{\gamma_j}+\frac{\alpha_j}{\gamma_j} \pi_j + \frac{\beta_j}{\gamma_j} \pi_{j-1}$$
For $k=0$ we have $x^0=\pi_0(x)=c_0(x^0) \pi_0$. For  $k=1$ from the three terms recurrence formula we have
$$x=x \pi_0= \frac{\pi_1}{\gamma_0} + \frac{\alpha_0}{\gamma_0}\pi_0 = c_1(x) \pi_1 +c_0(x) \pi_0 $$
In the inductive step the thesis holds for $k$ and we prove it for $k+1$. 
From the three terms recurrence formula we have
\begin{align*}
x^{k+1} =&  x x^k =  \sum_{j=0}^{k} c_j(x^k) x\pi_j \\
=& \sum_{j=0}^{k} c_j(x^k) \left( \frac{\pi_{j+1}}{\gamma_j}+\frac{\alpha_j}{\gamma_j} \pi_j + \frac{\beta_j}{\gamma_j} \pi_{j-1}\right )\\
=& \sum_{j=1}^{k+1}\frac{c_{j-1}(x^k)}{\gamma_{j-1}}  \pi_j 
       +  \sum_{j=0}^{k}c_j(x^k)\frac{\alpha_j}{\gamma_j}  \pi_j +  \sum_{j=0}^{k-1}c_{j+1}(x^k)\frac{\beta_{j+1}}{\gamma_{j+1}}  \pi_j \\
=& \sum_{j=1}^{k-1}\left(\frac{c_{j-1}(x^k)}{\gamma_{j-1}} 
      + c_j(x^k) \frac{\alpha_j}{\gamma_j} + c_{j+1}(x^k)\frac{\beta_{j+1}}{\gamma_{j+1}} \right )  \pi_j \\ 
&+ \frac{c_{k-1}(x^k)}{\gamma_{k-1}} \pi_k 
      +\frac{c_{k}(x^k)}{\gamma_{k}} \pi_{k+1}\\
&+ \frac{c_0(x^k)\alpha_0}{\gamma_0} \pi_0 +  \frac{c_{k}(x^k)\alpha_k}{\gamma_{k}} \pi_k + \frac{c_1(x^k)\beta_1}{\gamma_1} \pi_0 \\
=& \sum_{j=1}^{k-1} c_j(x^{k+1}) \pi_j +c_{k+1}(x^{k+1}) \pi_{k+1} \\
&+  \left(\frac{c_{k-1}(x^k)}{\gamma_{k-1}} +\frac{c_k(x^k) \alpha_k}{\gamma_k} \right) \pi_k \\ &+ \left(\frac{c_0(x^k)\alpha_0}{\gamma_0} + \frac{c_1(x^k) \beta_1}{\gamma_1} \right) \pi_0 
\end{align*}
This concludes the proof of the first part of the theorem. 

To prove the second part we apply what we just proved and unfold the multiplication. 
Given $x^\alpha=x_1^{\alpha_1}\cdots x_d^{\alpha_d}$, the polynomial  $\pi_\alpha= \pi_{\alpha_1}(x_1) \cdots \pi_{\alpha_d}(x_d) $ is the product 
of $d$ univariate polynomials $\pi_{\alpha_j}$ each of degree $\alpha_j$ in $x_j$, $j=1,\ldots,d$. Clearly if $\alpha_j=0$ then $\pi_{\alpha_j}=1$ and $x_j$ does not divide $x^\alpha$. Furthermore we have
  \begin{eqnarray*}
\pi_\alpha=  \prod_{j=1}^d  \sum_{k =0}^{\alpha_j} d_{k}^{(j)} x_j^{k}  \end{eqnarray*}
We deduce that $\pi_\alpha$ is a linear combination of $x^\alpha$ and of the power products which divide $x^\alpha$, that is of power products $x^\beta$ with $\beta \le \alpha$ component wise.
 Vice versa, applying the first part of the theorem we have  
 \begin{eqnarray*}
x^\alpha=\prod_{k=1}^d x_k^{\alpha_k} =\prod_{k=1}^d\left[\sum_{j_k =0}^{\alpha_k} c_{j_k}(x_k^{\alpha_k}) \pi_{j_k}(x_k) \right] 
 \end{eqnarray*}
 and commuting product with sum shows that
$x^\alpha$ is a linear combination of products of $\pi_{\beta_i}(x_i)$ where 
$\beta=(\beta_1,\ldots,\beta_d)$ is such that $\beta \le \alpha$ component wise, that is
$x^\beta$ divides $x^\alpha$.
\qed   \end{proof}

\medskip
{\bf{Theorem~\ref{RemainderOverHermiteBasis}:}}
\begin{proof}
Recall that $\mathcal B$ and $\mathcal {OB}$ are  defined in terms of a common set $L$ of $d$-dimensional vectors with non-negative integer entries satisfying the property of `factor-closeness', that is if $(\alpha_1,\ldots,\alpha_d)\in L$ 
and if $\beta_i\leq \alpha_i$ for all $i=1,\ldots,d$ then $(\beta_1,\ldots,\beta_d)\in L$.

\begin{enumerate}
\item 
If  $x^\alpha\in \mathcal B$ for some $\alpha$, then
from Theorem~\ref{Correspondence_Monomials_HermitePolynomials} 
$x^\alpha =  \sum_{\beta \le \alpha} b_\beta \pi_\beta $ follows.
Since $\beta \le \alpha$ then $\beta \in L$ and so each  $\pi_\beta\in \mathcal{OB}$: we  have that  $x^\alpha$ belongs to $\operatorname{Span}(\mathcal{OB})$. 
The vice versa is proved analogously.

\item The matrix  $[\pi_\beta(z) ]_{z\in \design,  \beta \in L} $ is a square matrix since $ L$ has as many elements as $\design$ and has  full rank. The linear independence of the columns of this matrix 
follows from the fact that each linear combination of its  columns corresponds to a polynomial in 
$\operatorname{Span}(\mathcal {OB})$ which coincides with $\operatorname{Span}(\mathcal B)$. 

Any polynomial  $g\in G$  can be written as
\begin{eqnarray*}
g= x^\alpha - \sum_{\alpha >_\sigma \beta \in L } c_\beta x^\beta
\end{eqnarray*}
where $x^\alpha=LT(g)$ is a multiple by some $x_j$ of an element of $\mathcal B$. 
By Theorem~\ref{Correspondence_Monomials_HermitePolynomials}  we have
\begin{eqnarray*}
g= \sum_{\gamma \le \alpha} a_\gamma ^{(g)} \pi_\gamma -  \sum_{\alpha >_\sigma \beta \in L } c_\beta \sum_{\gamma \le \beta } d_\gamma^{(g)} \pi_\gamma 
\end{eqnarray*}
The polynomial $\pi_\alpha$ appears only in the first sum with the coefficient $1$.
For the other terms in the first sum observe that 
as $\gamma < \alpha$ then $\gamma \in L$ and  also $\gamma <_\sigma \alpha $.
Analogously, for the second sum consider  $ \gamma \le \beta < \alpha$; since $\beta \in L$ then  $\gamma \in L$ and since $\gamma<\alpha$ then  $\gamma<_\sigma \alpha$.
And so, with obvious notation, 
\begin{eqnarray*}
g= \pi_\alpha -\sum_{ \alpha >_\sigma \beta \in L } b_\beta \pi_\beta
\end{eqnarray*} 

Since $g(z)=0$ for $z\in \design$, then the vector $b=[b_\beta]_{\beta}$ of the coefficients in the identity above 
solves the linear system  $[\pi_\beta(z)]_{z\in \design, \alpha >_\sigma \beta \in L  }b = [ \pi_\alpha(z)]_{z\in \design}$. 
Furthermore, since $[\pi_\beta(z)]_{z\in \design, \alpha >_\sigma \beta \in L  }$ is a full rank matrix, then $b$ is the unique solution of such a system.

\item Let $p^*= \sum_{\beta \in L} a_\beta \pi_\beta $ be the polynomial whose coefficients are   the solution of the linear system 
\[ [\pi_\beta(z)]_{z\in \design, \beta \in L}\,\, a  =  [p(z)] _{z\in \design}. \]
Such a polynomial obviously interpolates the values $p(z)$, $z\in \design$, and, since the columns of $[\pi_\beta(z)]_{z\in \design, \beta \in L}  $ are the evaluation vectors of the elements of $\mathcal {OB}$ at $\design$, it belongs to  $\operatorname{Span}(\mathcal {OB})$. 
We conclude that  $p^*$ is the unique polynomial belonging to $\operatorname{Span}(\mathcal {OB})$ which interpolates the values  $p(z)$, $z\in \design$.
\end{enumerate}
\qed   \end{proof}

{\bf Theorem~\ref{mio}:}
\begin{proof}
As $G$ is a basis of  $\mathcal I(\design)$, then for every  $p\in \mathcal I(\design)$  and $g\in G$
there exist $q_g\in \reals[x]$ such that  $p= \sum_{g\in G} q_g g $. Since by linearity
\begin{eqnarray*}
\expect_{\lambda} \left( \sum_{g\in G} q_g g \right)  = \sum_{g\in G}  \expect_{\lambda} \left( q_g g \right) 
\end{eqnarray*}
the thesis follows once we show that, for each $g\in G$ and $x^{\alpha}=LT(g)$
\begin{eqnarray*}
\expect_{\lambda} \left( q_g g \right) =  
 \| \pi_{\alpha} \|_\lambda^2 c_{\alpha}(q_g) 
 - \sum_{\alpha >_\sigma \beta \in L} c_\beta(q_g) c_\beta(g)\|\pi_\beta\|_\lambda^2 
\end{eqnarray*}
holds. From Equation~(\ref{strutturaq}) we have
\begin{equation*}
q_g g =\sum_{ \beta }        c_\beta(q_g)  \pi_\beta g  
\end{equation*}
and we substitute the Fourier expansion of $g$  given in Theorem~\ref{RemainderOverHermiteBasis} 
\begin{equation*}
g= \pi_{\alpha} - \sum_{\alpha >_\sigma \beta \in L} c_\beta(g) \pi_\beta
\end{equation*}
In computing the expectation we use the fact that $ \expect_{\lambda} \left( \pi_h \pi_k  \right) =0$ 
for different $h$ and $k$. Then the expectation of  $c_\beta(q_g)  \pi_\beta g$ vanishes if $\beta >_\sigma \alpha$ or 
$\beta <_\sigma \alpha$, $\beta \notin L$, the expectation of $c_{\alpha}(q_g)  \pi_{\alpha} g$
gives $ \| \pi_{\alpha} \|_\lambda^2 c_\alpha(q_g)$ and, if $\alpha >_\sigma \beta \in L$, the expectation of  $c_\beta(q_g)  \pi_\beta g$  gives 
$-c_\beta(q_g) c_\beta(g)\|\pi_\beta\|_\lambda^2 $. 
\qed   \end{proof}

\section*{Acknowledgments}
\label{sec:acknowledgments}
 G. Pistone is supported by de Castro Statistics Initiative, Collegio Carlo Alberto, Moncalieri Italy.  E. Riccomagno worked on this paper while visiting the Department of Statistics, University of Warwick, and the Faculty of Statistics at TU-Dortmund on a DAAD grant. Financial support  is gratefully acknowledged. The authors thank Prof. H. P. Wynn, Prof. G. Monegato (Politecnico di Torino) and Prof. Dr. Hans Michael M\"oller (Technische Universit\"at -- Dortmund) for their useful suggestions.


\begin{bibdiv}
\begin{biblist}

\bib{CocoaSystem}{misc}{
      author={{CoCoA}Team},
       title={{{\hbox{\rm C\kern-.13em o\kern-.07em C\kern-.13em o\kern-.15em
  A}}}: a system for doing {C}omputations in {C}ommutative {A}lgebra},
        note={Available at \/ {\tt http://cocoa.dima.unige.it}},
}

\bib{MR2290010}{book}{
      author={Cox, David},
      author={Little, John},
      author={O'Shea, Donal},
       title={Ideals, varieties, and algorithms},
     edition={Third},
      series={Undergraduate Texts in Mathematics},
   publisher={Springer},
     address={New York},
        date={2007},
}

\bib{DrtonSullivantSturmfelsLectureNotes}{book}{
      author={Drton, Mathias},  
      author={Sturmfels, Bernd}, 
      author={Sullivant, Seth},
       title={Lectures on Algebraic Statistics},
   publisher={Birkh\"auser},
     address={Basel},
        date={2009},
}

\bib{MR1772046}{article}{
      author={Fontana, Roberto},
      author={Pistone, Giovanni},
      author={Rogantin, Maria~Piera},
       title={Classification of two-level factorial fractions},
        date={2000},
        ISSN={0378-3758},
     journal={J. Statist. Plann. Inference},
      volume={87},
      number={1},
       pages={149\ndash 172},
}

\bib{MR2061539}{book}{
      author={Gautschi, Walter},
       title={Orthogonal polynomials: computation and approximation},
      series={Numerical Mathematics and Scientific Computation},
   publisher={Oxford University Press},
     address={New York},
        date={2004},
}

\bib{MR2640515}{book}{
      editor={Gibilisco, Paolo},
      editor={Riccomagno, Eva},
      editor={Rogantin, Maria~Piera},
      editor={Wynn, Henry~P.},
       title={Algebraic and geometric methods in statistics},
   publisher={Cambridge University Press},
     address={Cambridge},
        date={2010},
}

\bib{MR1335234}{book}{
      author={Malliavin, Paul},
       title={Integration and probability},
      series={Graduate Texts in Mathematics},
   publisher={Springer-Verlag},
     address={New York},
        date={1995},
      volume={157},
}


\bib{MR680050}{incollection}{
      author={M{\"o}ller, H.~M.},
      author={Buchberger, B.},
       title={The construction of multivariate polynomials with preassigned
  zeros},
        date={1982},
   booktitle={Computer algebra ({M}arseille, 1982)},
      series={Lecture Notes in Comput. Sci.},
      volume={144},
   publisher={Springer},
     address={Berlin},
       pages={24\ndash 31},
}

\bib{peccatitaqqu}{book}{
      author={Peccati, Giovanni},
      author={Taqqu, Murad S.},
       title={Wiener Chaos: Moments, Cumulants and Diagrams},
   publisher={Bocconi \& Springer-Verlag},
   address={Italia},
        date={2011}, 
}

\bib{MR2332740}{book}{
      author={Pistone, Giovanni},
      author={Riccomagno, Eva},
      author={Wynn, Henry~P.},
       title={Algebraic statistics},
      series={Monographs on Statistics and Applied Probability},
   publisher={Chapman \& Hall/CRC, Boca Raton, FL},
        date={2001},
}

\bib{MR1808151}{article}{
      author={Schoutens, Wim},
       title={Orthogonal polynomials in {S}tein's method},
        date={2001},
        ISSN={0022-247X},
     journal={J. Math. Anal. Appl.},
      volume={253},
      number={2},
       pages={515\ndash 531},
}

\bib{StroudSecrets}{book}{
    AUTHOR = {Stroud, A. H.},
    AUTHOR = {Secrest, Don},
     TITLE = {Gaussian quadrature formulas},
 PUBLISHER = {Prentice-Hall Inc.},
   ADDRESS = {Englewood Cliffs, N.J.},
      YEAR = {1966},
}

\bib{MR1768956}{article}{
      author={Xu, Yuan},
       title={Polynomial interpolation in several variables, cubature formul{\ae},
  and ideals},
        date={2000},
        ISSN={1019-7168},
     journal={Adv. Comput. Math.},
      volume={12},
      number={4},
       pages={363\ndash 376},
}

\bib{MR0907119}{article}{
   author={M{\"o}ller, H. Michael},
   title={On the construction of cubature formulae with few nodes using
   Groebner bases},
   conference={
      title={Numerical integration},
      address={Halifax, N.S.},
      date={1986},
   },
   book={
      series={NATO Adv. Sci. Inst. Ser. C Math. Phys. Sci.},
      volume={203},
      publisher={Reidel},
      place={Dordrecht},
   },
   date={1987},
   pages={177--192},
}
	
\end{biblist}
\end{bibdiv}


\end{document}